\documentclass[12,a4paper,reqno]{amsart}
\usepackage[UKenglish]{babel}
\usepackage[T1]{fontenc}
\usepackage[utf8]{inputenc}
\usepackage{verbatim}
\usepackage{palatino}
\usepackage{amsmath}
\usepackage{mathabx}
\usepackage{amssymb}
\usepackage{amsthm}
\usepackage{amsfonts}
\usepackage{graphicx}
\usepackage{esint}
\usepackage{xcolor}
\usepackage{mathtools}
\usepackage{overpic}
\DeclarePairedDelimiter\ceil{\lceil}{\rceil}

\usepackage[colorlinks = true,  citecolor = {blue!60!black},backref={page}]{hyperref}

\author{Tuomas Orponen, Nicolas de Saxc\'e, and Pablo Shmerkin}
\title{On the Fourier decay of multiplicative convolutions}
\address{Department of Mathematics and Statistics\\ University of Jyv\"askyl\"a,
P.O. Box 35 (MaD)\\
FI-40014 University of Jyv\"askyl\"a\\
Finland}
\email{tuomas.t.orponen@jyu.fi}
\address{CNRS, LAGA, Université Sorbonne Paris Nord, 99 avenue J.-B. Clément, 93430 Villetaneuse, FRANCE}
\email{desaxce@math.univ-paris13.fr}
\address{Department of Mathematics\\
The University of British Columbia\\
1984 Mathematics Road, Vancouver, BC\\
Canada} \email{pshmerkin@math.ubc.ca}
\date{\today}
\thanks{T.O. is supported by the Academy of Finland via the project
\emph{Approximate incidence geometry}, grant no. 355453, and by the European Research Council (ERC) under the European Union's Horizon 2020 research and innovation programme (project MUSING, grant agreement no. 101087499).
\thanks{P.S. is supported by an NSERC Discovery Grant.}
}

\newcommand{\R}{\mathbb{R}}

\newcommand{\N}{\mathbb{N}}

\newcommand{\Z}{\mathbb{Z}}

\newcommand{\spt}{\operatorname{spt}}
\newcommand{\Hd}{\dim_{\mathrm{H}}}

\newcommand{\dist}{\operatorname{dist}}

\newcommand{\abs}[1]{\lvert#1\rvert}
\newcommand{\Abs}[1]{\left\lvert#1\right\rvert}
\newcommand{\norm}[1]{\lVert#1\rVert}

\newcommand{\bracket}[1]{\langle#1\rangle}
\newcommand{\ff}{\boxtimes}
\newcommand{\mm}{\boxminus}
\newcommand{\pp}{\boxplus}

\newcommand{\brho}{\boldsymbol\rho}

\def\Barint_#1{\mathchoice
          {\mathop{\vrule width 6pt height 3 pt depth -2.5pt
                  \kern -8pt \intop}\nolimits_{#1}}%
          {\mathop{\vrule width 5pt height 3 pt depth -2.6pt
                  \kern -6pt \intop}\nolimits_{#1}}%
          {\mathop{\vrule width 5pt height 3 pt depth -2.6pt
                  \kern -6pt \intop}\nolimits_{#1}}%
          {\mathop{\vrule width 5pt height 3 pt depth -2.6pt
                  \kern -6pt \intop}\nolimits_{#1}}}

\numberwithin{equation}{section}

\theoremstyle{plain}
\newtheorem{thm}[equation]{Theorem}
\newtheorem*{"thm"}{"Theorem"}

\newtheorem{lemma}[equation]{Lemma}

\newtheorem{cor}[equation]{Corollary}
\newtheorem{proposition}[equation]{Proposition}

\theoremstyle{definition}

\newtheorem{definition}[equation]{Definition}

\theoremstyle{remark}
\newtheorem{remark}[equation]{Remark}

\newtheorem{example}[equation]{Example}

\newcommand{\nref}[1]{(\hyperref[#1]{#1})}

\DeclareMathSymbol{\intop}  {\mathop}{mathx}{"B3}

\begin{document}

\begin{abstract} We prove the following. Let $\mu_{1},\ldots,\mu_{n}$ be Borel probability measures on $[-1,1]$ such that $\mu_{j}$ has finite $s_j$-energy  for certain indices $s_{j} \in (0,1]$ with $s_{1} + \ldots + s_{n} > 1$. Then, the multiplicative convolution of the measures $\mu_{1},\ldots,\mu_{n}$ has power Fourier decay: there exists a constant $\tau = \tau(s_{1},\ldots,s_{n}) > 0$ such that
    \begin{displaymath}
        \left| \int e^{-2\pi i \xi \cdot x_{1}\cdots x_{n}} \, d\mu_{1}(x_{1}) \cdots \, d\mu_{n}(x_{n}) \right| \leq |\xi|^{-\tau}
    \end{displaymath}
    for sufficiently large $|\xi|$. This verifies a suggestion of Bourgain from 2010. We also obtain a quantitative Fourier decay exponent under a stronger assumption on the exponents $s_{j}$.
\end{abstract}

\maketitle

\section{Introduction}

In 2010, Bourgain~\cite[Theorem 6]{Bo2} proved the following remarkable Fourier decay property for multiplicative convolutions of Frostman measures on the real line.

\begin{thm}[Fourier decay for multiplicative convolutions]
    \label{bourgain_decay}
    For all $s>0$, there exist $\epsilon>0$ and $n\in\Z_+$ such that the following holds for every $\delta>0$ sufficiently small.\\
    If $\mu$ is a probability measure on $[-1,1]$ satisfying
    \[
        \forall r\in[\delta,\delta^{\epsilon}],\quad \sup_{a\in [-1,1]} \mu(B(a,r)) < r^s
    \]
    then for all $\xi\in\R$ with $\delta^{-1} \leq \abs{\xi} \leq 2\delta^{-1}$,
    \begin{equation}\label{fd}
        \int e^{2\pi i\xi x_1\dots x_n} d\mu(x_1)\dots d\mu(x_n) \leq \delta^{-\epsilon}.
    \end{equation}
\end{thm}

This result found striking applications in the Fourier decay of fractal measures and resulting spectral gaps for hyperbolic surfaces \cite{BourgainDyatlov17, SahlstenStevens22}. It was recently generalised to higher dimensions by Li \cite{Li21}.

At the end of the introduction of \cite{Bo2}, Bourgain proposes to study the optimal relation between $s$ and $n$. Our goal here is to show that, as suggested by Bourgain, Theorem \ref{bourgain_decay} holds under the condition $n>1/s$, which is optimal up to the endpoint, as we shall see in Example~\ref{hi} below.

The statement we obtain applies more generally to multiplicative convolutions of different measures, and our proof also allows us to replace the Frostman condition by a slightly weaker condition.
Precisely, for a finite Borel measure $\mu$ on $\R$, given $s\in(0,1]$ and $\delta>0$, the \emph{$s$-energy} of $\mu$ is defined as
\begin{equation} \label{eq:energy}
    I_s(\mu) := \iint \abs{x-y}^{-s}\,d\mu(x)\,d\mu(y),
\end{equation}
or, in terms of the Fourier transform, as
\begin{equation} \label{eq:energy-fourier}
    I_s(\mu) = c_s \int \abs{\widehat{\mu}(\xi)}^2 \abs{\xi}^{-s}\,d\xi,
\end{equation}
where $c_s>0$ is a constant. We refer the reader to \cite[Lemma 12.12]{mattila} for the last equality, and to \cite[Chapter 8]{mattila} for the basic properties of the energy of a measure. As in Bourgain's theorem, we shall be mostly interested in the properties of measures up to some fixed small scale $\delta$; for that reason, we also define the $s$-energy of $\mu$ at scale $\delta$ by
\begin{displaymath}
    I^{\delta}_{s}(\mu) = I_s(\mu_\delta),
\end{displaymath}
where $\mu_\delta=\mu*P_\delta$ is the regularisation of $\mu$ at scale $\delta$, and $\{P_\delta\}_{\delta \in (0,1]}$ an approximate unity of the usual form $P(x) = \delta^{-1}P(x/\delta)$, where $P \in C^{\infty}(\R)$ satisfies $\mathbf{1}_{B(1/2)} \leq P \lesssim \mathbf{1}_{B(1)}$.

The main result of the present article is the following.

\begin{thm}[Fourier decay of multiplicative convolutions under optimal energy condition]
    \label{main}
    Let $n \geq 2$, and $\{s_{j}\}_{j = 1}^{n} \subset (0,1]$ such that $\sum s_{j} > 1$.
    Then, there exist $\delta_{0},\epsilon,\tau \in (0,1]$, depending only on the parameters above, such that the following holds for $\delta \in (0,\delta_{0}]$.
    Let $\mu_{1},\ldots,\mu_{n}$ be Borel probability measures on $[-1,1]$ satisfying the energy conditions
    \begin{equation}\label{frostman}
        I_{s_{j}}^{\delta}(\mu_{j}) \leq \delta^{-\epsilon}, \qquad 1 \leq j \leq n.
    \end{equation}
    Then, for all $\xi$ satisfying $\delta^{-1} \leq |\xi| \leq 2\delta^{-1}$,
    \begin{equation}\label{form9}
        \left| \int e^{-2\pi i \xi  x_{1}\dots x_{n}} \, d\mu_{1}(x_{1}) \dots \, d\mu_{n}(x_{n}) \right| \leq |\xi|^{-\tau}.
    \end{equation}
\end{thm}

\begin{remark}
    It is not difficult to check that the Frostman condition $\mu(B(a,r))\leq r^s$ from Bourgain's Theorem~\ref{bourgain_decay} is stronger than the assumption on the $s$-energy at scale $\delta$ used above.
    The reader is referred to Lemma~\ref{frostman-energy} for a detailed argument.
\end{remark}

\begin{remark}
    The values of the parameters $\delta_{0},\epsilon$ and  $\tau> 0$ stay bounded away from $0$ as long as the $s_j$ stay uniformly bounded away from $0$, and $\sum_{j} s_{j} > 1$ stays bounded away from $1$, and $n$ ranges in a bounded subset of $\N$. Unfortunately, we are not able to provide explicit bounds for these parameters.
\end{remark}

The following corollary is immediate:
\begin{cor} \label{cor:main}
    Let $n\ge 2$, and $\{ s_j\}_{j=1}^{n}\subset (0,1]$ such that $\sum s_j>1$.
    There exists $\tau=\tau(n,\{s_j\})>0$ such that the following holds. Let $\mu_1,\dots,\mu_n$ be Borel probability measures on $\R$ such that $I_{s_j}(\mu_j)<+\infty$. Then there is $C=C(\{\mu_j\})>0$ such that
    \begin{equation}\label{decay}
        \left| \int e^{-2\pi i \xi  x_{1}\dots x_{n}} \, d\mu_{1}(x_{1}) \dots \, d\mu_{n}(x_{n}) \right| \leq C\cdot |\xi|^{-\tau}, \quad\xi\in\R.
    \end{equation}
\end{cor}

As a further corollary, we obtain the following sum-product statement:
\begin{cor} \label{cor:sets}
    Let $A_1,\ldots, A_n \subset \R$ be Borel sets such that $\sum_{j=1}^n \Hd A_j > 1$. Then the additive subgroup generated by $A_1\cdots A_n$ is $\R$.
\end{cor}
\begin{proof}
    Let $\mu_j$ be an $s_j$-Frostman measure on $A_j$, where $\sum_{j=1}^n s_j>1$. Write $\mu_1\ff\cdots\ff\mu_n$ for the image of the measure $\mu_1\times\dots\times\mu_n$ under the product map $(x_1,\dots,x_n)\mapsto x_1\dots x_n$; this measure is supported on $A_1\cdots A_n$. Then the Fourier decay condition~\eqref{decay} implies that additive convolution powers of $\mu_1\ff\cdots\ff\mu_n$ become absolutely continuous with respect to the Lebesgue measure on $\R$, with arbitrarily smooth densities. In particular, a sumset of the product set $A_1\cdots A_n$ must contain a non-empty interval, and hence the additive subgroup generated by $A_1\cdots A_n$ is $\R$.
\end{proof}

In the example below, we show that the condition $\sum \Hd A_j>1$ in Corollary \ref{cor:sets} is optimal up to the endpoint; hence so is the condition $\sum s_j>1$ in Theorem~\ref{main}.

\begin{example}
    \label{hi}
    Given $s\in(0,1)$ and an increasing sequence of integers $(n_k)_{k\geq 1}$, define a subset $H_s$ in $\R$ by
    \[
        H_s = \left\{x\in [0,1]\ : \ \forall k\geq 1,\ \dist(x,n_k^{-s}\Z)\leq n_k^{-1}\right\}.
    \]
    If $(n_k)_{k\geq 1}$ grows fast enough, then both $H_s$ and the additive subgroup it generates will have Hausdorff dimension $s$. See e.g. \cite[Example 4.7]{Falconer14}.

    Now assume that the parameters $s_1,\dots,s_n$ satisfy $\sum s_i<1$.
    Fixing $s_i'>s_i$ such that one still has $\sum s_i'<1$, Frostman's lemma yields probability measures $\mu_i$ supported on $H_{s_i'}$ and satisfying $\mu_i(B(a,r))<r^{s_i}$ for all $r>0$ sufficiently small.
    However, since the support $A_i$ of $\mu_i$ satisfies $A_i\subset H_{s_i'}$, one has
    \[
        A_1\dots A_n \subset \left\{x\in[0,1]\ : \forall k\geq 1,\ \dist\left(x,n_k^{-\sum s_i'}\Z\right)\leq n\cdot n_k^{-1}\right\}.
    \]
    This shows that the subgroup generated by $A_1\dots A_n$ has dimension bounded above by $\sum s_i'<1$ and so is not equal to $\R$.
    So the measure $\mu_1\ff\cdots\ff\mu_n$ cannot have polynomial Fourier decay.
\end{example}

Allowing more multiplicative convolutions, one can get an improved exponential lower bound for the Fourier decay exponent.
This is the content of our second main result.
\begin{thm}[Exponential lower bound for Fourier decay exponent]
    \label{second}
    For every $\sigma \in (0,1]$, there exist $\epsilon = \epsilon(\sigma) > 0$, $\delta_{0} = \delta_{0}(\sigma) > 0$, and an absolute constant $C \geq 1$, such that the following holds for all $\delta \in (0,\delta_{0}]$.
    Let $n\geq C\sigma^{-1}$, and assume $\mu_1,\dots,\mu_n$ are probability measures on $[1,2]$ satisfying
    \[
        I_\sigma^\delta(\mu_i) \leq \delta^{-\epsilon}.
    \]
    Then,
    \begin{displaymath}
        \left| \int e^{-2\pi i \xi x_{1}\dots x_{n}} \, d\mu_{1}(x_{1}) \dots \, d\mu_{n}(x_{n}) \right| \leq |\xi|^{-\exp(-C\sigma^{-1})}, \quad \delta^{-1} \leq |\xi| \leq 2\delta^{-1}.
    \end{displaymath}
\end{thm}

Analogues of Theorem~\ref{main} and \ref{second} in the prime field setting were obtained by Bourgain in \cite{bourgain_primefield}, and our proofs follow a similar general strategy, based on sum-product estimates and flattening for additive-multiplicative convolutions of measures, although the details differ significantly. A small variant of Example~\ref{hi} above shows that there exist compact sets $A$ and $B$ in $\R$ such that the additive subgroup $\bracket{AB}$ generated by the product set $AB$ satisfies
\(
\Hd\bracket{AB} \leq \Hd A + \Hd B.
\)
Conversely, it was shown in~\cite{OrponenShmerkin23} as a consequence of the discretised radial projection theorem~\cite{osw_radialprojection} that for Borel sets $A,B \subset \R$, one has
\begin{equation}\label{2ab}
    \Hd(AB+AB - AB-AB) \geq \min\{\Hd A + \Hd B,1\}.
\end{equation}
The main ingredient in the proof of Theorem~\ref{main} is a discretised version of this inequality; the precise statement is given below as Proposition~\ref{prop:expansion} and is taken from \cite[Proposition~3.7]{OrponenShmerkin23}. It can be understood as a precise version of the discretised sum-product theorem, which allows us to improve on the strategy used by Bourgain in \cite{Bo2} and obtain Fourier decay of multiplicative convolutions under optimal energy conditions. Before turning to the detailed proof, let us give a general idea of the argument.

\subsection*{Notation}

We use $\lesssim$ to hide absolute multiplicative constants.

We fix for the rest of the article a standard, $L^1$-normalized approximate identity
\[
    \{P_{\delta}\}_{\delta > 0} = \{\delta^{-1}P(\cdot/\delta)\}_{\delta > 0}.
\]
We take $P$ to be radially decreasing and to satisfy $\mathbf{1}_{[-1/2,1/2]} \leq P \leq \mathbf{1}_{[-1,1]}$.

Given a measure $\mu$ on $\R$, recall that we write $\mu_\delta$ for the density of $\mu$ at scale $\delta$, or equivalently, $\mu_\delta=\mu*P_\delta$.

Below, we shall use both additive and multiplicative convolution of measures.
To avoid any confusion, we write $\mu\pp\nu$, $\mu\mm\nu$ and $\mu\ff\nu$ to denote the image of $\mu\times\nu$ under the maps $(x,y)\mapsto x+y$, $(x,y)\mapsto x-y$, and $(x,y)\mapsto xy$, respectively. Similarly, we denote additive and multiplicative $k$-convolution powers of measures by $\mu^{\pp k}$ and $\mu^{\ff k}$, respectively.

The push-forward of a Borel measure $\mu$ on the real line under a Borel map $g:\R\to\R$ is denoted $g_{\sharp}\mu$, that is,
\[
    \int f\,d(g_{\sharp}\mu) = \int f\circ g\,d\mu.
\]

\subsection*{Sketch of proof of Theorem~\ref{main}}

The $n=2$ case of Theorem~\ref{main} is classical and already appears in Bourgain's paper~\cite[Theorem~7]{Bo2}: Let $\mu$ and $\nu$ be two probability measures on $[-1,1]$ such that $\norm{\mu_\delta}_2^2\leq\delta^{-1+s}$ and $\norm{\nu_\delta}_2^2\leq\delta^{-1+t}$ (these can be seen as "single-scale" versions of $I_s^{\delta}(\mu)\lessapprox 1$, $I_t^{\delta}(\nu)\lessapprox 1$). Then the multiplicative convolution $\mu\ff\nu$ satisfies
\[
    \abs{\widehat{\mu\ff\nu}(\xi)}\lesssim\delta^{\frac{s+t-1}{2}},\quad \delta^{-1}\leq\abs{\xi}\leq 2\delta^{-1}.
\]
For the reader's convenience we record the detailed argument below, see Section~\ref{sec:n2}.

We want to use induction to reduce to this base case.
To explain the induction step, we focus on the case $n=3$.
The main point is to translate equation~\eqref{2ab} into a flattening statement for additive-multiplicative convolutions of measures.
For simplicity, assume we knew that if $\mu$ and $\nu$ are probability measures on $[-1,1]$, then the measure
\[
    \eta := (\mu\ff\nu)\pp(\mu\ff\nu) \mm (\mu\ff\nu)\mm(\mu\ff\nu)
\]
satisfies, for $\epsilon>0$ arbitrarily small,
\begin{equation}\label{ab_measures}
    \norm{\eta_\delta}_2^2 \leq \delta^{1-\epsilon} \norm{\mu_\delta}_2^2\norm{\nu_\delta}_2^2.
\end{equation}
(Note that this is a close analogue of \eqref{2ab} for $L^2$-dimensions of measures at scale $\delta$.)
If $\mu_1$, $\mu_2$ and $\mu_3$ satisfy $\norm{(\mu_i)_\delta}_{2}^2\leq\delta^{-1+s_i}$ for some parameters $s_i$ with $s_1+s_2+s_3>1$, we apply the above inequality to $\mu_1$ and $\mu_2$ to obtain
\[
    \norm{\eta_\delta}_2^2 \leq \delta^{-1-\epsilon+s_1+s_2},
\]
where $\eta=(\mu_1\ff\mu_2)\pp(\mu_1\ff\mu_2) \mm (\mu_1\ff\mu_2)\mm(\mu_1\ff\mu_2)$.
If $\epsilon$ is chosen small enough, we have $(s_1+s_2-\epsilon)+s_3>1$, and so we may apply the $n=2$ case to the measures $\eta$ and $\mu_3$ to get, for $\delta^{-1}<\abs{\xi}<2\delta^{-1}$,
\[
    \abs{\widehat{\eta\ff\mu_3}(\xi)}<\delta^{\frac{s_1+s_2+s_3-\epsilon-1}{2}}.
\]
To conclude, one observes from the Cauchy-Schwarz inequality that for any two probability measures $\mu$ and $\nu$, one always has  $\abs{\widehat{\mu\ff\nu}(\xi)}^2\leq \abs{\widehat{(\mu\mm\mu)\ff\nu}(\xi)}$. This elementary observation applied twice yields
\[
    \abs{(\mu_1\ff\mu_2\ff\mu_3)^{\wedge}(\xi)}^4 \leq \abs{\widehat{\eta\ff\mu_3}(\xi)}^4 < \delta^{\frac{1}{2}(s_1+s_2+s_3-\epsilon-1)}
\]
which is the desired Fourier decay, with parameter $\tau=\frac{1}{8}(s_1+s_2+s_3-\epsilon-1)$.

Unfortunately, as shown in Example \ref{ex:L2-not-enough} below, the assumptions on the $L^2$-norms of $\mu$ and $\nu$ are not sufficient to ensure inequality~\eqref{ab_measures} in general.
One also needs some kind of non-concentration condition on $\mu$ and $\nu$, and for that purpose we use the notion of energy of the measure at scale $\delta$, which gives information on the behaviour of the measure at all scales between $\delta$ and $1$.
The precise statement we use for the induction is given as Lemma~\ref{lemma3} below.
It is also worth noting that to obtain the correct bound on the energy, we need to use a large number $k$ of additive convolutions, whereas $k=4$ was sufficient in the analogous statement~\eqref{2ab} for Hausdorff dimension of sum-product sets.
We do not know whether Lemma~\ref{lemma3} holds for $k=4$, or even for $k$ bounded by some absolute constant.

\subsection*{Sketch of proof of Theorem~\ref{second}}
In the proof of Theorem~\ref{main}, the fact that the number $k$ of additive convolutions used in Lemma~\ref{lemma3} is not bounded by an absolute constant prevents us from getting an explicit lower bound for the decay exponent $\tau$ in Theorem~\ref{main}.
To overcome this issue and derive Theorem~\ref{second}, we use another flattening statement, Lemma~\ref{flattening} below.
This lemma uses only one additive convolution, and this gives better control on the Fourier decay exponent.
The drawback is that the increase in dimension is not as good as the one given by Lemma~\ref{lemma3}.
Roughly speaking, the measure $\eta$ will only satisfy $\dim \eta\geq\dim \mu + \frac{\dim\nu}{C}$ for some absolute constant $C$ instead of $\dim\eta\geq\dim\mu+\dim\nu-\epsilon$.
However, this is sufficient to get Fourier decay when the convolution product is long enough.
(The bound is not as good as the one in Theorem~\ref{main}, but comparable within an absolute multiplicative constant.)\\\
The proof of the flattening statement Lemma~\ref{flattening} relies on a new quantitative estimate for projections of discretised sets under weak non-concentration assumptions on the set of projections, Theorem~\ref{thm:rescy}, which may have independent interest.
Theorem~\ref{thm:rescy} is derived from a recent projection theorem of Ren and Wang~\cite{RenWang23}, using rescaling arguments.

\bigskip
We conclude this introduction by an example showing that for $n \geq 3$, the assumption $I_{s_{j}}^{\delta}(\mu_{j}) \leq\delta^{-\epsilon}$ in Theorems~\ref{main} and \ref{second} cannot be replaced by the "single-scale" $L^{2}$-bound $\|\mu_{\delta}\|_{2}^{2} \leq \delta^{s_{j} - 1-\epsilon}$.
This is mildly surprising, because the situation is opposite in the case $n = 2$, as shown by Proposition~\ref{basecase} below.
We only write down the details of the example in the case $n = 3$, but it is straightforward to generalise to $n \geq 3$.

\begin{example}  \label{ex:L2-not-enough}
    For every $s \in (0,\tfrac{1}{2})$ and $\delta_{0} > 0$ there exists a scale $\delta \in (0,\delta_{0}]$ and a Borel probability measure $\mu = \mu_{\delta,s}$ on $[1,2]$ with the following properties:
    \begin{itemize}
        \item $\|\mu\|_{2}^{2} \sim \delta^{s - 1}$.
        \item $|\widehat{\mu\ff\mu\ff\mu}(\delta^{-1})| \sim 1$.
    \end{itemize}
    The building block for the construction is the following.
    For $r \in 2^{-\N}$, and a suitable absolute constant $c > 0$, let $\mathcal{I} = \mathcal{I}_{r}$ be a family of $r^{-1}$ intervals of length $cr$, centred around the points $r\Z \cap [0,1]$. Then, if $c > 0$ is small enough, we have
    \begin{displaymath}
        \cos(2\pi x/r) \geq \tfrac{1}{2}, \qquad \forall x \in \cup \mathcal{I}.
    \end{displaymath}
    Consequently, if $\brho$ is any probability measure supported on $\cup \mathcal{I}$, then $|\widehat{\brho}(r^{-1})| \geq \tfrac{1}{2}$.

    Fix $s \in (0,\tfrac{1}{2})$, $\delta > 0$, and let $\rho$ be the uniform probability measure on the intervals $\mathcal{I}_{\delta^{s}}$.
    As we just discussed, $|\widehat{\rho}(\delta^{-s})| \sim 1$.
    Next, let $\mu = \mu_{\delta,s}$ be a rescaled copy of $\rho$ inside the interval $[1,1 + \delta^{1 - s}] \subset [1,2]$. More precisely, $\mu = \tau_{\sharp} \lambda_{\sharp} \rho$, where $\tau(x) = x + 1$ and $\lambda(x) = \delta^{1 - s}x$.
    Now $\mu$ is a uniform probability measure on a collection of $\delta^{-s}$ intervals of length $\delta$, and consequently $\|\mu\|_{2}^{2} \sim \delta^{s - 1}$.

    We next investigate the Fourier transform of $\mu\ff\mu\ff\mu$.
    Writing $\mu_0 := \lambda_{\sharp} \rho$, we have
    \begin{displaymath}
        \widehat{\mu\ff\mu\ff\mu}(\delta^{-1}) = \iiint e^{-2\pi i \delta^{-1}(x + 1)(y + 1)(z + 1)} \, d\mu_0(x) \, d\mu_0(y) \, d\mu_0(z).
    \end{displaymath}
    We expand
    \begin{displaymath} \delta^{-1}(x + 1)(y + 1)(z + 1) = \delta^{-1}xyz + \delta^{-1}(xy + xz + yz) + \delta^{-1}(x + y + z) + \delta^{-1}. \end{displaymath}
    Now the key point: since $x,y,z \in \spt \mu_0 \subset [0,\delta^{1 - s}]$, we have
    \begin{displaymath}
        |\delta^{-1} xyz| \leq \delta^{2 - 3s}\quad\text{and}\quad |\delta^{-1}(xy + xz + yz)| \lesssim \delta^{1 - 2s}.
    \end{displaymath}
    Since $s < \tfrac{1}{2}$, both exponents $2 - 3s$ and $1 - 2s$ are strictly positive and consequently,
    \begin{displaymath} e^{-2\pi i \delta^{-1}(x + 1)(y + 1)(z + 1)} = e^{-2\pi i \delta^{-1}(x + y + z + 1)} + o_{\delta \to 0}(1). \end{displaymath}

    Using this, and also that $\hat{\mu}_0(\xi) = \widehat{\rho}(\delta^{1 - s}\xi)$, we find
    \begin{align*}
        \widehat{\mu\ff\mu\ff\mu}(\delta^{-1}) & = e^{-2\pi i \delta^{-1}}\iiint e^{-2\pi i \delta^{-1}(x + y + z)} \, d\mu_0(x) \, d\mu_0(y) \, d\mu_0(z) + o_{\delta \to 0}(1) \\
                                               & = e^{-2\pi i \delta^{-1}} (\widehat{\rho}(\delta^{-s}))^{3} + o_{\delta \to 0}(1).
    \end{align*}
    In particular, $|\widehat{\mu\ff\mu\ff\mu}(\delta^{-1})| \sim 1$ for $\delta > 0$ sufficiently small.
\end{example}

If we allow $\mu$ to be supported on $[-1,1]$, as in Theorem~\ref{main}, an even simpler example is available, namely $\mu=\delta^{-1+s}\mathbf{1}_{[0,c\delta^{1-s}]}$, where $s<2/3$.
Then $\mu\ff\mu\ff\mu$ is supported on $[0,c^3\delta^{3-3s}]\subset [0,c\delta]$, so it satisfies $\abs{\widehat{\mu\ff\mu\ff\mu}(\delta^{-1})}\gtrsim 1$ if $c>0$ is chosen small enough.

\section{The base case \texorpdfstring{$n=2$}{n=2}}
\label{sec:n2}

In the $n=2$ case, Theorem~\ref{main} is proved by a direct elementary computation.
In fact, to obtain the desired Fourier decay, one only needs an assumption on the $L^2$-norms of the measures at scale $\delta$.

\begin{proposition}[Base case $n=2$]
    \label{basecase}
    Let $\delta \in (0,1]$ and let $\mu,\nu$ be Borel probability measures on $[-1,1]$ and $s,t \in [0,1]$ such that
    \[
        \norm{\mu_\delta}_2^2\leq\delta^{-1+s}
        \quad\mbox{and}\quad
        \norm{\nu_\delta}_2^2\leq\delta^{-1+t}.
    \]
    Then, for all $\xi$ with $\delta^{-1} \leq |\xi| \leq 2\delta^{-1}$,
    \[
        \left| \iint e^{-2\pi i \xi \cdot xy} \, d\mu(x) \, d\nu(y) \right|
        \lesssim \delta^{\frac{s+t-1}{2}}.
    \]
\end{proposition}

\begin{remark}
    If $\mu$ and $\nu$ are equal to the normalized Lebesgue measure on balls of size $\delta^{1-s}$ and $\delta^{1-t}$, respectively, the assumptions of the proposition are satisfied.
    In that case, the multiplicative convolution $\mu\ff\nu$ is supported on a ball of size $\delta^{1-\max(s,t)}$, so the Fourier decay cannot hold for $\abs{\xi}\leq\delta^{-1+\max(s,t)}$.
\end{remark}

The above proposition is an easy consequence of the lemma below, which is essentially \cite[Theorem 7]{Bo2}, except that we keep slightly more careful track of the constants. We include the proof for completeness.

\begin{lemma}\label{lemma4}
    Let $\delta \in (0,1]$, and let $\mu,\nu$ be Borel probability measures on $[-1,1]$ with
    \begin{displaymath} A := \int_{|\xi| \leq 2\delta^{-1}} |\hat{\mu}(\xi)|^{2} \, d\xi \quad \text{and} \quad B := \int_{|\xi| \leq 2\delta^{-1}} |\hat{\nu}(\xi)|^{2} \, d\xi. \end{displaymath}
    Then, for all $\xi$ with $1\leq\abs{\xi}\leq\delta^{-1}$,
    \begin{equation}\label{form3}
        \left| \iint e^{-2\pi i \xi \cdot xy} \, d\mu(x) \, d\nu(y) \right| \lesssim \sqrt{AB/|\xi|} + \delta.
    \end{equation}
\end{lemma}
\begin{proof}
    Let $\varphi \in C^{\infty}_{c}(\R)$ be an auxiliary function with the properties $\mathbf{1}_{[-1,1]} \leq \varphi \leq \mathbf{1}_{[-2,2]}$ (thus $\varphi \equiv 1$ on $\spt \mu$) and $\widehat{\varphi} \geq 0$.
    Fixing $1 \leq |\xi| \leq \delta^{-1}$, the left-hand side of \eqref{form3} can be estimated by
    \begin{align*}
        \left| \iint e^{-2\pi i \xi \cdot xy} \, d\mu(x) \, d\nu(y) \right|
         & = \left| \int \widehat{\varphi \mu}(y\xi) \, d\nu(y)\right| \leq \iint \widehat{\varphi}(x - y \xi) |\hat{\mu}(x)| \, dx \, d\nu(y) \\
         & = \int |\hat{\mu}(x)| \left( \int \widehat{\varphi}(x - y \xi) \, d\nu(y) \right) \, dx.
    \end{align*}
    We split the right-hand side as the sum
    \begin{displaymath} \int_{|x| \leq 2\delta^{-1}} |\hat{\mu}(x)|  \left( \int \widehat{\varphi}(x - y\xi) \, d\nu(y) \right) \, dx + \iint_{|x| \geq 2\delta^{-1}} |\hat{\mu}(x)| \widehat{\varphi}(x - y\xi) \, dx \, d\nu(y) =: I_{1} + I_{2}. \end{displaymath}
    For the term $I_{2}$, we use that $\widehat{\varphi}(x - y\xi) \lesssim |x - y\xi|^{-2}$, $|\hat{\mu}(x)| \leq 1$, and $\nu(\R) = 1$:
    \begin{displaymath}
        I_{2} \lesssim \max_{y \in [-1,1]} \int_{|x| \geq 2\delta^{-1}} \frac{dx}{|x - y\xi|^{2}} \lesssim \int_{|x| \geq \delta^{-1}} \frac{dx}{|x|^{2}} \lesssim \delta.
    \end{displaymath}
    For the term $I_{1}$, we first use the Cauchy-Schwarz inequality and the definition of $A$ to deduce
    \begin{displaymath}
        I_{1} \leq \sqrt{A}\left( \int \left[ \int \widehat{\varphi}(x - y\xi) \, d\nu(y) \right]^{2} \, dx \right)^{1/2}.
    \end{displaymath}
    Finally, for the remaining factor, assume $\xi > 0$ without loss of generality, and write $\widehat{\varphi}(x - y\xi) = \widehat{\varphi_{\xi}}(x/\xi - y)$, where $\varphi_{\xi} = \xi^{-1}\varphi(\cdot/\xi)$. With this notation, and by Plancherel's formula,
    \begin{align*}
        \int \left[ \int \widehat{\varphi}(x - y\xi) \, d\nu(y) \right]^{2} \, dx
         & = \int (\widehat{\varphi_{\xi}} \ast \nu)(x/\xi)^{2} \, dx
        = \xi \int (\widehat{\varphi_{\xi}} \ast \nu)(z)^{2} \, dz    \\
         & = \xi \int \varphi_{\xi}(u)^{2}|\hat{\nu}(u)|^{2} \, du
        \lesssim \xi^{-1} \int_{\spt \varphi_{\xi}} |\hat{\nu}(u)|^{2} \, du.
    \end{align*}
    Finally, recall that $\spt \varphi \subset [-2,2]$, so $\spt \varphi_{\xi} \subset [-2\xi,2\xi] \subset [-2\delta^{-1},2\delta^{-1}]$.
    This shows that $I_{1} \lesssim \sqrt{AB/\xi}$, and the proof of \eqref{form3} is complete.
\end{proof}

\begin{proof}[Proof of Proposition~\ref{basecase}]
    Observe that by Plancherel's formula
    \[
        A = \int_{\abs{\xi}\leq 4\delta^{-1}} \abs{\hat{\mu}(\xi)}^2\,d\xi
        \leq \norm{\mu_{\frac{\delta}{10}}}_2^2
        \lesssim \norm{\mu_\delta}_2^2
        \leq \delta^{-1+s}
    \]
    and similarly
    \[
        B = \int_{\abs{\xi}\leq 4\delta^{-1}} \abs{\hat{\mu}(\xi)}^2\,d\xi
        \lesssim \delta^{-1+t}.
    \]
    So Lemma~\ref{lemma3} applied at scale $\delta/2$ implies that for $\delta^{-1}\leq\abs{\xi}\leq 2\delta^{-1}$,
    \begin{align*}
        \left| \iint e^{-2\pi i \xi \cdot xy} \, d\mu(x) \, d\nu(y) \right|
         & \lesssim \sqrt{AB/|\xi|} + \delta  \\
         & \lesssim \delta^{\frac{s+t-1}{2}}.
    \end{align*}
\end{proof}

\section{Dimension and energy of additive-multiplicative convolutions}

This section is the central part of the proof of Theorem~\ref{main}.
Its goal is to derive Lemma~\ref{lemma3} below, whose statement can be qualitatively understood in the following way: If $\mu$ and $\nu$ are two Borel probability measures on $\R$ with respective dimensions $s$ and $t$, then there exists some additive convolution of $\mu\ff\nu$ with dimension at least $s+t-\epsilon$, where $\epsilon>0$ can be arbitrarily small.
The precise formulation in terms of the energies of the measures at scale $\delta$ will be essential in our proof of Fourier decay for multiplicative convolutions.

\begin{lemma}\label{lemma3}
    For all $s,t \in (0,1]$ with $s + t \leq 1$, and for all $\kappa > 0$, there exist $\epsilon = \epsilon(s,t,\kappa) > 0$, $\delta_{0} = \delta_{0}(s,t,\kappa,\epsilon) > 0$, and $k_{0} = k_{0}(s,t,\kappa) \in \N$ such that the following holds for all $\delta \in (0,\delta_{0}]$ and $k \geq k_{0}$. Let $\mu,\nu$ be Borel probability measures on $[-1,1]$ satisfying
    \begin{equation}\label{form5} I^{\delta}_{s}(\mu) \leq \delta^{-\epsilon} \quad \text{and} \quad I^{\delta}_{t}(\nu) \leq \delta^{-\epsilon}. \end{equation}
    Then, with $\Pi := (\mu\mm\mu)\ff(\nu\mm\nu)$, we have
    \begin{displaymath}
        I^{\delta}_{s + t}(\Pi^{\pp k}) \leq \delta^{-\kappa}.
    \end{displaymath}
    Moreover, the value of $k_{0}$ stays bounded as long as $\min\{s,t\} > 0$ stays bounded away from zero.
\end{lemma}

The main component of the proof of Lemma \ref{lemma3} will be a combinatorial result from \cite{OrponenShmerkin23} (itself based on the main resul of \cite{osw_radialprojection}), which we will apply in the following form. Let $\abs{E}_\delta$ denote the smallest number of $\delta$-balls required to cover $E$.

\begin{lemma}
    \label{lemma2}
    For all $s,t \in (0,1]$ with $s + t \leq 1$, and for all $\kappa > 0$, there exist $\epsilon = \epsilon(s,t,\kappa) > 0$ and $\delta_{0} = \delta_{0}(s,t,\kappa,\epsilon) > 0$ such that the following holds for all $\delta \in (0,\delta_{0}]$ and $k \geq 2$.
    Let $\mu,\nu$ be Borel probability measures on $[-1,1]$ satisfying
    \begin{equation}\label{form20}
        I_{s}^{\delta}(\mu) \leq \delta^{-\epsilon}
        \quad \text{and} \quad
        I_{t}^{\delta}(\nu) \leq \delta^{-\epsilon}.
    \end{equation}
    Let $\Pi := (\mu \mm \mu) \ff (\nu \mm \nu)$ and assume that $E\subset\R$ is a set with $\Pi^{\pp k}(E) \geq \delta^{\epsilon}$.
    Then,
    \begin{equation}\label{form13}
        \abs{E}_\delta \geq \delta^{-s - t + \kappa}.
    \end{equation}
\end{lemma}
Since this lemma does not explicitly appear in \cite{OrponenShmerkin23}, we now briefly explain how to derive it from the results of that paper.
Recall that a Borel measure $\mu$ on $\R$ is said to be \emph{$(s,C)$-Frostman} if it satisfies
\[
    \mu(B(x,r))\leq Cr^s \quad\text{for all }x\in\R, r>0.
\]
The precise statement we shall need is \cite[Proposition 3.7]{OrponenShmerkin23}, see also \cite[Remark 3.11]{OrponenShmerkin23}, which reads as follows.

\begin{proposition}
    \label{prop:expansion}
    Given $s,t\in (0,1]$ and $\sigma\in [0,\min\{s+t,1\})$, there exist $\epsilon = \epsilon(s,t,\sigma) > 0$ and $\delta_{0} = \delta_{0}(s,t,\sigma,\epsilon) > 0$  such that the following holds for all $\delta \in (0,\delta_{0}]$.\\
    Let $\mu_{1},\mu_{2}$ be $(s,\delta^{-\epsilon})$-Frostman probability measures, let $\nu_{1},\nu_{2}$ be $(t,\delta^{-\epsilon})$-Frostman probability measures, all four measures supported on $[-1,1]$, and let $\brho$ be an $(s + t,\delta^{-\epsilon})$-Frostman probability measure supported on $[-1,1]^{2}$.
    Then there is a set $\mathbf{Bad} \subset \R^{4}$ with
    \[
        (\mu_{1} \times \mu_{2} \times \nu_{1} \times \nu_{2})(\mathbf{Bad}) \le \delta^{\epsilon},
    \]
    such that for every $(a_1,a_2,b_1,b_2)\in \R^{4} \, \setminus \, \mathbf{Bad}$ and every subset $G \subset \R^{2}$ satisfying $\brho(G) \geq \delta^{\epsilon}$, one has
    \begin{equation} \label{eq:robust-sum-prod}
        \Abs{\{(b_1- b_2)a + (a_1- a_2) b : (a,b) \in G\}}_\delta \ge \delta^{-\sigma}.
    \end{equation}
\end{proposition}

The derivation of Lemma~\ref{lemma2} from Proposition~\ref{prop:expansion} is relatively formal; it mostly uses the link between the Frostman condition and the energy at scale $\delta$, and the pigeonhole principle to construct large fibres in product sets.
Let us first record an elementary statement about the energy at scale $\delta$ of a Frostman measure.
Below, we also use the energy $I_s^\delta(\mu)$ for a measure supported on $\R^d$, with $d\geq 2$.
The definition is the same as \eqref{eq:energy} in the one-dimensional case, except that $\abs{\cdot}$ denotes the Euclidean norm on $\R^d$. We also let $I_s^\delta(\mu)=I_s(\mu*P_\delta)$, where $\{P_\delta\}_{\delta>0}=\{\delta^{-d}P(\cdot/\delta)\}_{\delta>0}$ is an approximate identity, with $P\in C^\infty(\R^d)$ radially decreasing and satisfying $\mathbf{1}_{B_(0,1/2)} \leq P \leq \mathbf{1}_{B(0,1)}$.

\begin{lemma}[Frostman condition and $s$-energy]
    \label{frostman-energy} Fix $C\geq 1$, $s\in(0,d)$, and $\epsilon \in (0,\tfrac{1}{2}]$. Then the following holds for all $\delta > 0$ small enough. Let $\mu$ be a probability measure on $B(1) \subset \R^{d}$.
    \begin{enumerate}
        \item If $\mu$ satisfies $\mu(B(x,r))\leq Cr^s$ for all $x\in \R^{d}$ and all $r\in[\delta,\delta^\epsilon]$, then $I_s^\delta(\mu) \leq C\delta^{-d \epsilon }$.
        \item Conversely, if $I_s^\delta(\mu)\leq\delta^{-\epsilon}$, there exists a set $A$ such that $\mu(A)\geq 1-(\log1/\delta)\delta^{\epsilon}$ and for every $r\in[\delta,1]$, $\mu|_A(B(x,r))\leq \delta^{-2\epsilon}r^s$.
    \end{enumerate}
\end{lemma}
\begin{proof}
    To begin, observe that
    \begin{equation} \label{eq:local-energy}
        \int\abs{x-y}^{-s}\,d\mu_\delta(y) = s\int \mu_\delta(B(x,r)) r^{-s-1}\,dr.
    \end{equation}
    Assume first that $\mu$ satisfies the Frostman condition $\mu(B(x,r))\leq Cr^s$ for $r\in[\delta,\delta^\epsilon]$. It is not difficult to check that the measure $\mu_{\delta}$ with density $\mu \ast P_{\delta}$ satisfies
    \begin{displaymath} \mu_{\delta}(B(x,r)) \lesssim \begin{cases} C\delta^{s - d}r^{d}, & 0 < r \leq \delta, \\ Cr^{s}, & \delta \leq r \leq \delta^{\epsilon}, \\ 1, & r \geq \delta^{\epsilon}. \end{cases} \end{displaymath}
    By splitting the integral in \eqref{eq:local-energy} into the three intervals $(0,\delta]$, $[\delta,\delta^{\epsilon}]$, and $[\delta^{\epsilon},\infty)$, we find that
    \begin{displaymath}
        \int\abs{x-y}^{-s}\,d\mu_\delta(y)  \lesssim_{s}C\delta^{-s\epsilon}.
    \end{displaymath}
    Since $\mu_{\delta}$ is a probability measure, and $s < d$, this implies $I_{s}^{\delta}(\mu) \leq C\delta^{-d\epsilon}$ for $\delta > 0$ small enough.

    For the converse, we observe that \eqref{eq:local-energy} and a change of variables yield
    \[
        I_s^\delta(\mu) = s\iint \mu_\delta(B(x,r)) r^{-s}\,\frac{dr}{r}\,d\mu(x)
        = (\log 2)\cdot s \iint \mu_\delta(B(x,2^{-u})) 2^{su}\,du\,d\mu(x).
    \]
    If $I_s^\delta(\mu)\leq\delta^{-\epsilon}$, letting
    \[
        E_u=\{x\ : \ 2^{su}\mu_\delta(B(x,2^{-u}))> \delta^{-2\epsilon}\},
    \]
    one gets $\mu(E_u)\leq\delta^{\epsilon}$.
    So, for $E=\bigcup E_u$, where $u=0,1,\dots,\lfloor\log1/\delta\rfloor$,
    we find
    \(
    \mu(E) \leq (\log1/\delta)\delta^{\epsilon}.
    \)
    Thus, letting $A=\R^{d} \, \setminus \, E$, one indeed has
    \[
        \mu(A)\geq 1-(\log1/\delta)\delta^{\epsilon}
    \]
    and for all $x$ in $A$, for all $r\in[\delta,1]$, $\mu(B(x,r))\lesssim \delta^{-2\epsilon}r^s$.
\end{proof}

\begin{proof}[Proof of Lemma~\ref{lemma2}]
    First of all, we may assume that $k = 2$, since if $k > 2$, we may write
    \begin{displaymath}
        \Pi^{\pp k}(E) = \int \Pi^{\pp 2}(E - x_{3} - \ldots - x_{k}) \, d\Pi(x_{3}) \cdots \, d\Pi(x_{k}),
    \end{displaymath}
    and in particular there exists a vector $(x_{3},\ldots,x_{k})$ such that $\Pi^{\pp 2}(E - x_{3} - \ldots - x_{k}) \geq \delta^{\epsilon}$.
    After this, it suffices to prove \eqref{form13} with $E - x_{3} - \ldots - x_{k}$ in place of $E$.

    Second, we may assume that the measures $\mu,\nu$ satisfy the Frostman conditions
    \begin{displaymath}
        \mu(B(x,r)) \leq \delta^{-6\epsilon}r^{s} \quad \text{and} \quad \nu(B(x,r)) \leq \delta^{-6\epsilon}r^{t}
    \end{displaymath}
    for $\delta \leq r \leq 1$ and all $x \in \R$.
    Indeed, since $I_s^\delta(\mu)\leq\delta^{-3\epsilon}$, Lemma~\ref{frostman-energy} shows that there exists a Borel set $A \subset \R$ of measure $\mu(A) \geq 1 - \delta^{2\epsilon}$ with the property $(\mu|_{A})(B(x,r)) \leq \delta^{-6\epsilon}r^{s}$ for all $x\in\R$ and all $r\in[\delta,1]$.
    Similarly, we may find a Borel set $B \subset \R$ of measure $\nu(B) \geq 1 - \delta^{2\epsilon}$ with the property $(\nu|_{B})(B(x,r)) \leq \delta^{-6\epsilon}r^{t}$.
    Now, we still have $\overline{\Pi}^{\pp 2}(E) \geq \tfrac{1}{2}\delta^{\epsilon}$, where
    \begin{displaymath}
        \overline{\Pi} := (\mu|_{A} \mm \mu|_{A}) \ff (\nu|_{B} \mm \nu|_B).
    \end{displaymath}
    Therefore, we may proceed with the argument, with $\mu,\nu$ replaced by $\mu|_{A},\nu|_{B}$.

    Let us rewrite the condition $\Pi^{\pp 2}(E) \geq \delta^{\epsilon}$ as $(\mu \times \nu)^{4}(G_8) \geq \delta^{\epsilon}$, where
    \begin{displaymath}
        G_8 := \{(a_{1},b_{1},\ldots,a_{4},b_{4}) \in \R^{8} : (a_{1} - a_{2})(b_{3} - b_{4}) + (b_{1} - b_{2})(a_{3} - a_{4}) \in E\}.
    \end{displaymath}
    In particular, there exists a subset $G_{6} \subset \R^{6}$ of measure $(\mu \times \nu)^{3}(G_{6}) \geq \delta^{2\epsilon}$ such that for every $(a_1,b_1,a_2,b_2,a_3,b_3)$ in $G_6$, one has $(\mu \times \nu)(G_{2}) \geq \delta^{2\epsilon}$, where
    \begin{equation}\label{form14}
        G_{2} := \{(a_{4},b_{4}) \in \R^{2} : (a_{1},b_{1},\ldots,a_{4},b_{4}) \in G_{8}\}.
    \end{equation}
    Next, we plan to apply Proposition \ref{prop:expansion}. To make this formally correct, let us "freeze" two of the variables, say $(a_{3},b_{3})$: more precisely, fix $(a_{3},b_{3})$ in such a way that $(\mu \times \nu)^{2}(G_{4}) \geq \delta^{2\epsilon}$, where
    \begin{displaymath}
        G_4 := \{(a_{1},b_{1},a_{2},b_{2}) \in \R^{4} : (a_{1},b_1,a_{2},b_{2},a_3,b_{3}) \in G_{6}\}.
    \end{displaymath}
    If $\epsilon$ is chosen small enough in terms of $s$, $t$ and $\sigma := s+t-\kappa$, we may apply Proposition \ref{prop:expansion} with $\mu_{1} = \mu_{2} = \mu$ and $\nu_{1} = \nu_{2} = \nu$, and $\brho = \mu \times \nu$, using $6\epsilon$ instead of $\epsilon$.
    Then, one has
    \begin{displaymath}
        (\mu \times \nu)^{2}(\mathbf{Bad})
        < \delta^{6\epsilon} \leq (\mu \times \nu)^{2}(G_{4}),
    \end{displaymath}
    and \eqref{eq:robust-sum-prod} holds for all $(a_{1},b_{1},a_2,b_{2}) \in \R^{4} \, \setminus \, \mathbf{Bad}$.
    Consequently, we may find a $4$-tuple $(a_{1},b_1,a_{2},b_{2}) \in G_{4} \, \setminus \, \mathbf{Bad}$, and eventually a $6$-tuple
    \begin{displaymath}
        (a_{1},b_1,a_{2},b_{2},a_3,b_{3}) \in G_{6}
    \end{displaymath}
    such that whenever $G \subset \R^{2}$ is a Borel set with $(\mu \times \nu)(G) = \brho(G) \geq \delta^{6\epsilon}$, then
    \begin{align*}
         & \Abs{\{(a_{1} - a_{2})(b_{3} - b_{4}) + (b_{1} - b_{2})(a_{3} - a_{4}) : (a_{4},b_{4}) \in G\}}_\delta                                 \\
         & = \Abs{\{(a_{1} - a_{2})b_{4} + (b_{1} - b_{2})a_{4} : (a_{4},b_{4}) \in G\}}_\delta \geq \delta^{-\sigma} = \delta^{-s - t + \kappa}.
    \end{align*}
    In particular, by \eqref{form14}, this can be applied to the set $G := G_{2}$, and the conclusion is that
    \begin{equation}\label{form15}
        \Abs{\{(a_{1} - a_{2})(b_{3} - b_{4}) + (b_{1} - b_{2})(a_{3} - a_{4}) : (a_{4},b_{4}) \in G_{2}\}}_\delta \geq \delta^{-s - t + \kappa}.
    \end{equation}
    However, since $(a_{1},b_{1},a_{2},b_{2},a_{3},b_{3}) \in G_{6}$, we have $(a_{1},b_{1},\ldots,a_{4},b_{4}) \in G_8$ for all $(a_{4},b_{4}) \in G_{2}$, and consequently
    \begin{displaymath}
        \forall (a_{4},b_{4}) \in G_{2},\qquad (a_{1} - a_{2})(b_{3} - b_{4}) + (b_{1} - b_{2})(a_{3} - a_{4}) \in E.
    \end{displaymath}
    Therefore, \eqref{form15} implies \eqref{form13}.
\end{proof}

We now want to go from the combinatorial conclusion of Lemma~\ref{lemma2} to the more measure theoretic statement of Lemma~\ref{lemma3} involving energies at scale $\delta$.
For that, our strategy is similar in flavour to the one used by Bourgain and Gamburd~\cite{bg_su2} to derive their flattening lemma, decomposing the measures into dyadic level sets.

\begin{proof}[Proof of Lemma \ref{lemma3}] Let $\{P_{r}\}_{r > 0} = \{r^{-1}P(\cdot/r)\}_{r > 0}$ be a radially decreasing approximate identity satisfying $\mathbf{1}_{[-1/2,1/2]} \leq P \leq \mathbf{1}_{[-1,1]}$. We claim that
    \begin{equation}\label{form33} P(x/r) \leq P(x'/(4r)), \qquad |x - x'| \leq r. \end{equation}
    Indeed, if $P(x/r) \neq 0$, then $|x| \leq r$, hence $|x'| \leq 2r$, and $P(x'/(4r)) = 1$. We further claim:
    \begin{equation}\label{form34}
        P\left(\frac{x - y}{4r} \right) \lesssim \frac{1}{5r} \int_{B(x,5r)} P\left(\frac{x' - y}{r} \right) \, dx', \qquad x,y \in \R. \end{equation}
    Indeed, if $P((x - y)/4r) \neq 0$, then $|x - y| \leq 4r$, so $B(y,r/2) \subset B(x,5r)$. Now, \eqref{form34} follows by noting that $P((x' - y)/r) = 1$ for $x' \in B(y,r/2)$.

    As a final preliminary, we record that if $\sigma$ is a finite Borel measure on $\R^{d}$, and $0 < s < d$, then
    \begin{equation}\label{form35} I^{r}_{s}(\sigma) \lesssim I^{\delta}_{s}(\sigma), \qquad \delta \leq r \leq 1. \end{equation}
    To see this, note, as a first step, that $I^{r}_{s}(\sigma) \leq I_{s}(\sigma)$. This follows from the Fourier-analytic expression of the energy \ref{eq:energy-fourier}, and noting that $\|\hat{P}\|_{\infty} = \|P\|_{1} = 1$. As a second step, note that $P_{r} \lesssim P_{r} \ast P_{\delta}$ for $\delta \leq r \leq 1$, so also $\sigma_{r} \lesssim \sigma_{r} \ast P_{\delta}$. Consequently, $I^{r}_{s}(\sigma) \lesssim I^{r}_{s}(\sigma_{\delta}) \leq I_{s}(\sigma_{\delta}) = I_{s}^{\delta}(\sigma)$ by the first step of the proof.

    We then begin the proof in earnest. We abbreviate $\Pi^{\pp l}_{r} := (\Pi^{\pp l}) \ast P_{r}$. The goal will be to show that if $k \geq 1$ is sufficiently large (depending on $s,t,\kappa$), then, for all $r \in [\delta,1]$,
    \begin{equation}\label{form16}
        J_{r}(k) := \norm{\Pi_{r}^{\pp 2^k}}_{2} \leq \delta^{-\kappa/2}r^{(s + t - 1)/2}.
    \end{equation}
    This implies in a standard manner (using for example \cite[Lemma 12.12]{mattila}, Plancherel and a dyadic frequency decomposition) that
    \[
        I_{s + t}^\delta(\Pi^{\pp 2^k}) \lesssim \delta^{-2\kappa}.
    \]
    Note that the sequence $\{J_{r}(k)\}_{k \in \N}$ is decreasing in $k$, since by Young's inequality
    \begin{displaymath}
        J_{r}(k + 1) = \norm{\Pi^{\pp 2^k}\pp\Pi_r^{\pp 2^k}}_2
        \leq \norm{\Pi^{\pp 2^k}}_1\cdot\norm{\Pi_r^{\pp 2^k}}_2
        = \norm{\Pi_r^{\pp 2^k}}_2
        = J_r(k).
    \end{displaymath}
    Therefore, in order to prove \eqref{form16}, the value of $k$ may depend on $r$, as long as it is uniformly bounded in terms of $s,t,\kappa$.
    Eventually, the maximum of all possible values for $k$ will work for all $\delta \leq r \leq 1$.

    Let us start by disposing of large $r$, i.e. $r\geq\delta^{\kappa/2}$.
    For that, we have the trivial bound
    (recalling also that we assumed $s+t\leq 1$)
    \[
        J_{r}(k) \leq J_{r}(0) \lesssim r^{-1} \leq \delta^{-\kappa/2}r^{(s + t - 1)/2}.
    \]

    So, it remains to treat the case $r \in[\delta, \delta^{\kappa/2}]$.
    We now fix such a scale $r$.
    By the pigeonhole principle, given a small parameter $\epsilon \in (0,\tfrac{\kappa}{4})$ to be fixed later (the choice will roughly be determined by applying Lemma \ref{lemma2} to the parameters $s,t,\kappa$), there exists $k \lesssim 1/\epsilon$, depending on $r$, such that
    \begin{equation}\label{form17}
        \|\Pi_{r}^{\pp 2^{k+1}}\|_{2} = J_{r}(k + 1) \geq r^{\epsilon}J_{r}(k) = r^{\epsilon} \|\Pi_{r}^{\pp 2^k}\|_{2}.
    \end{equation}
    This index $k$ is fixed for the rest of the argument, so we will not display it in (all) subsequent notation.
    We may assume that $\|\Pi_{r}^{\pp 2^{k}}\|_{2} \geq 1$, otherwise \eqref{form16} is clear.

    Let $\mathcal{D}_{r}$ be the dyadic intervals of $\R$ of length $r$. For each $I \in \mathcal{D}_{r}$, we set
    \begin{displaymath} a_{I} := \sup_{x \in I} \Pi_{r}^{\pp 2^{k}}(x). \end{displaymath}
    Next, we define the collections
    \begin{displaymath} \mathcal{A}_{j} := \begin{cases} \{I \in \mathcal{D}_{r} : a_{I} \leq 1\}, & j = 0, \\ \{I \in \mathcal{D}_{r} : 2^{j - 1} < a_{I} \leq 2^{j}\}, & j \geq 1. \end{cases} \end{displaymath}
    We also define the sets $A_{j} := \cup \mathcal{A}_{j}$; note that the sets $A_{j}$ are disjoint for distinct $j$ indices.
    Since $\Pi$ is a probability measure, $\Pi_{r}^{\pp 2^{k}} \lesssim 1/r$ for all $k \geq 1$. Therefore $A_{j} = \emptyset$ for $j \geq C\log(1/r)$, and evidently
    \begin{equation}\label{form18a} \Pi_{r}^{\pp 2^{k}} \lesssim \sum_{j = 0}^{C\log(1/r)} 2^{j} \cdot \mathbf{1}_{A_{j}}. \end{equation}
    Conversely, we claim that
    \begin{equation}\label{form18b}
        \sum_{j = 1}^{C\log(1/r)} 2^{j} \cdot \mathbf{1}_{A_{j}} \lesssim \Pi^{\pp 2^{k}}_{4r}.
    \end{equation}
    To see this, fix $x \in A_{j}$ with $j \geq 1$, and let $I = I(x) \in \mathcal{D}_{r}$ be the dyadic $r$-interval containing $x$. Then $a_{I} \geq 2^{j - 1}$, which means that there exists another point $x' \in I$ with $\Pi_{r}^{\pp 2^{k}}(x') \geq 2^{j - 1}$. Then, using that $|(x' - y) - (x - y)| \leq r$ for all $y \in \R$,
    \begin{displaymath} \Pi_{4r}^{\pp 2^{k}}(x) = \frac{1}{4r} \int P\left(\frac{x - y}{4r} \right) \, d\Pi^{\pp 2^{k}}(y) \stackrel{\eqref{form33}}{\geq} \frac{1}{4r} \int P\left(\frac{x' - y}{r} \right) \, d\Pi^{\pp 2^{k}}(y) \geq \tfrac{1}{4} \Pi^{\pp 2^{k}}_{r}(x'), \end{displaymath}
    and consequently $\Pi_{4r}^{\pp 2^{k}}(x) \gtrsim 2^{j}$. This proves \eqref{form18b}.

    Based on \eqref{form18b} (and our hypothesis $\|\Pi_{r}^{\pp 2^{k}}\|_{2} \geq 1$ to treat the case $j = 0$) we may deduce, in particular, that
    \begin{equation}\label{form19}
        2^j \norm{\mathbf{1}_{A_j}}_2
        \lesssim \norm{\Pi_{4r}^{\pp 2^k}}_2 \stackrel{\eqref{form34}}{\lesssim} \norm{\Pi_{r}^{\pp 2^{k}}}_{2}, \qquad j \geq 0.
    \end{equation}
    (To see the second inequality, we use \eqref{form34} to deduce that $\mu_{4r}$ is bounded by the maximal function of $\mu_{r}$ for any probability measure $\mu$ on $\R$, in particular $\mu = \Pi^{\pp 2^{k}}$.)

    Next, using \eqref{form18a}, we may pigeonhole an index $j \geq 0$ and a set $A:=A_j$ with the property
    \begin{displaymath}
        \|\Pi_{r}^{\pp 2^{k+1}}\|_{2} \leq \|\Pi_{r}^{\pp 2^k} \mm \Pi_{r}^{\pp 2^k}\|_{2}
        \lesssim (\log1/r)\cdot 2^j\cdot\norm{\mathbf{1}_{A} \mm \Pi_{r}^{\pp 2^k}}_{2}.
    \end{displaymath}
    Since further, by Plancherel followed by Cauchy-Schwarz and Plancherel again,
    \begin{displaymath}
        \|\mathbf{1}_{A} \mm \Pi_{r}^{\pp 2^k}\|_{2}
        \leq \|\mathbf{1}_{A} \mm \mathbf{1}_{A}\|_{2}^{1/2}\cdot\|\Pi_{r}^{\pp 2^{k+1}}\|_{2}^{1/2}
    \end{displaymath}
    we deduce that
    \begin{align}
        r^{\epsilon}\|\Pi_{r}^{\pp 2^k}\|_{2}
        \stackrel{\eqref{form17}}{\leq} \|\Pi_{r}^{\pp 2^{k+1}}\|_{2}
         & \lesssim (\log1/r)^2\cdot 2^{2j}\cdot\norm{\mathbf{1}_{A}\mm\mathbf{1}_{A}}_2 \notag                                                                                                                               \\
         & \label{form21} \lesssim r^{-\epsilon}\cdot 2^{2j}\cdot\norm{\mathbf{1}_{A}}_1\norm{\mathbf{1}_{A}}_2 \stackrel{\eqref{form19}}{\lesssim} r^{-\epsilon}\cdot 2^j\cdot\norm{\mathbf{1}_A}_1\norm{\Pi_r^{\pp 2^k}}_2.
    \end{align}
    At this point we note that if $j = 0$, then the preceding inequality shows that $\|\Pi_{r}^{\pp 2^{k}}\|_{2} \lesssim r^{-2\epsilon}$, which is better than \eqref{form16}, since we declared that $\epsilon \leq \kappa/4$. So, we may and will assume that $j \geq 1$ in the sequel.

    In (i) below we combine \eqref{form21} and \eqref{form18b}, whereas in (ii) below we combine \eqref{form21} with $2^{j}\|\mathbf{1}_{A}\|_{1} \lesssim \|\Pi_{r}^{\pp 2^{k}}\|_{1} = 1$:
    \begin{itemize}
        \item[(i)] $r^{2\epsilon} \lesssim 2^j\norm{\mathbf{1}_A}_1 \lesssim \Pi_{4r}^{\pp 2^k}(A) \leq \Pi^{\pp 2^{k}}([A]_{4r})$, where $[A]_{4r}$ is the $4r$-neighbourhood of $A$,
        \item[(ii)] $r^{2\epsilon}\|\Pi_{r}^{\pp 2^k}\|_{2} \lesssim 2^j\norm{\mathbf{1}_A}_2 \lesssim \norm{\mathbf{1}_A}_1^{-1}\norm{\mathbf{1}_A}_2$ = $\norm{\mathbf{1}_A}_1^{-1/2}$.
    \end{itemize}
    Since $A$ is a union of intervals in $\mathcal{D}_{r}$, one has
    $\norm{\mathbf{1}_A}_1\sim r\abs{A}_r$, so item~(ii) yields
    \[
        \norm{\Pi_r^{\pp 2^k}}_2 \lesssim r^{-\frac{1}{2}-2\epsilon} \abs{A}_r^{-1/2}.
    \]
    On the other hand, since
    \begin{displaymath}
        I^{4r}_{s}(\mu) \stackrel{\eqref{form35}}{\lesssim} I^{\delta}_{s}(\mu) \leq \delta^{-\epsilon} \leq r^{-\epsilon/\kappa} \quad \text{and} \quad I^{4r}_{t}(\nu) \lesssim r^{-\epsilon/\kappa},
    \end{displaymath}
    Lemma~\ref{lemma2} applied at scale $4r$ (and recalling (i) above) shows that if $\epsilon$ is chosen small enough in terms of $s,t,\kappa$, then
    \[
        \abs{A}_r \geq r^{-s-t+\kappa}.
    \]
    We thus obtain what we claimed in \eqref{form16}:
    \begin{displaymath}
        \|\Pi_{r}^{\pp 2^k}\|_{2} \lesssim \delta^{-\kappa/2}r^{(s + t - 1)/2}.
    \end{displaymath}
    This completes the proof of Lemma \ref{lemma3}. \end{proof}

\section{The induction step}
\label{sec:induction}

The proof of Theorem~\ref{main} is by induction on $n$, starting from the $n=2$ case, already studied in Section~\ref{sec:n2}.
The induction step is based on the flattening results for additive-multiplicative convolutions developed in the previous section.
It will be essential in the argument to be able to switch the order of addition and multiplication. For that we record the following lemma, which is a simple application of the Cauchy-Schwarz inequality.

\begin{lemma}\label{order}
    Given two Borel probability measures $\mu$ and $\nu$ on $\R$, one has, for all $\xi$ in $\R$,
    \[
        \abs{\widehat{\mu\ff\nu}(\xi)}^2 \leq \big((\mu\mm\mu)\ff\nu\big)^\wedge(\xi).
    \]
\end{lemma}
\begin{proof}
    Writing the Fourier transforms explicitly, and applying the Cauchy-Schwarz (or Jensen's) inequality, one gets
    \[
        \Abs{\iint e^{-2\pi i \xi xy}\,d\mu x\,d\nu y}^2
        \leq \int\Abs{\int e^{-2\pi i\xi xy}\,d\mu x}^2\,d\nu y
        =  \iiint e^{-2\pi i\xi (x_1-x_2)y}\,d\mu x_1\,d\mu x_2\,d \nu y.
    \]
\end{proof}

\begin{proof}[Proof of Theorem~\ref{main}]
    For the base case $n=2$, we may apply Proposition~\ref{basecase}.
    Indeed, assuming $I_{s_i}^\delta(\mu_i)\leq\delta^{-\epsilon}$ for $i=1,2$, one has
    \[
        \norm{\mu_{i,\delta}}_2^2
        = \int \abs{\widehat{\mu_{i,\delta}}(\xi)}^2\,d\xi
        \lesssim \delta^{-1+s_i-\epsilon} \int \abs{\xi}^{1-s_i}\abs{\widehat{\mu_{i,\delta}}(\xi)}^2\,d\xi
        \overset{\eqref{eq:energy-fourier}}{\lesssim} \delta^{-1+s_i-2\epsilon}
    \]
    so if $s_1+s_2>1$, taking $\epsilon$ small enough to ensure $s_1+s_2-4\epsilon>1$, one finds, for every $\delta^{-1}\leq\abs{\xi}\leq 2\delta^{-1}$,
    \[
        \abs{\widehat{\mu_1\ff\mu_2}(\xi)} \leq \delta^{\frac{s_1+s_2-4\epsilon-1}{2}},
    \]
    which is the desired Fourier decay.

    Now let $n \geq 3$, and assume that we have already established the case $n - 1$ with the collection of parameters $\mathcal{S}_{n - 1} := \{s_{1} + s_{2},s_{3},\ldots,s_{n}\}$, and some constants
    \begin{equation}\label{form6}
        \epsilon_{n - 1}(\mathcal{S}_{n - 1}) > 0 \quad \text{and} \quad \delta_{0} := \delta_{0}(\mathcal{S}_{n - 1}) > 0.
    \end{equation}
    By decreasing $s_{1}$ and $s_{2}$ if needed, we assume that $s_{1} + s_{2} \leq 1$ in the sequel.

    Given $\xi$ with $\delta^{-1}\leq\xi\leq 2\delta^{-1}$, our goal is to bound
    \begin{displaymath}
        \mathcal{F}(\xi) := (\mu_1\ff\dots\ff\mu_n)^\wedge(\xi).
    \end{displaymath}
    Applying Lemma~\ref{order} twice, first with $\mu=\mu_1$ and $\nu=\mu_2\ff\dots\ff\mu_n$ and then with $\mu=\mu_2$ and $\nu=(\mu_1\mm\mu_1)\ff\mu_3\ff\dots\ff\mu_n$, yields
    \[
        \abs{\mathcal{F}(\xi)}^4\leq (\Pi\ff\mu_3\ff\dots\ff\mu_n)^\wedge(\xi),
    \]
    where $\Pi=(\mu_1\mm\mu_1)\ff(\mu_2\mm\mu_2)$. Using the same lemma again $k$ times (and noting  that $\Pi$ is symmetric around the origin), we further get
    \[
        \abs{\mathcal{F}(\xi)}^{2^{k+2}}\leq (\Pi^{\pp 2^k}\ff\mu_3\ff\dots\ff\mu_n)^\wedge(\xi).
    \]
    Lemma~\ref{lemma3} applied with constants $s := s_{1}$ and $t := s_{2}$, and $\kappa := \epsilon_{n - 1} := \epsilon_{n - 1}(\mathcal{S}_{n - 1})$, shows that if $\epsilon = \epsilon(s_{1},s_{2},\epsilon_{n - 1}) > 0$ is sufficiently small, $k = k(s_{1},s_{2},\epsilon_{n - 1})$ is sufficiently large, and $\mu_1$, $\mu_2$ satisfy $I_{s_j}^\delta(\mu_j)\leq\delta^{-\epsilon}$ for $j=1,2$, then
    \begin{displaymath}
        I_{s_{1} + s_{2}}^{\delta}(\Pi^{\pp 2^k}) \leq \delta^{-\kappa} = \delta^{-\epsilon_{n - 1}}.
    \end{displaymath}
    We apply our induction hypothesis to the collection of $n - 1$ probability measures
    \begin{displaymath}
        \{\bar{\mu}_{1},\ldots,\bar{\mu}_{n - 1}\} = \{\Pi^{\pp 2^k},\mu_{3},\ldots,\mu_{n}\}
    \end{displaymath}
    with exponents $\{s_{1} + s_{2},s_{3},\ldots,s_{n}\}$ to get
    \begin{displaymath}
        |\mathcal{F}(\xi)|^{2^{k+2}} \leq |\xi|^{-\epsilon_{n - 1}}.
    \end{displaymath}
    (To be precise, since the measure $\Pi^{\pp 2^k}$ is not supported on $[-1,1]$ but on $[-2^{k+2},2^{k+2}]$, so one rather needs to consider the rescaled measure $\bar{\mu}_1=(2^{-k-2})_*\Pi^{\pp 2^k}$, which satisfies $I_{s_1+s_2}^{\delta}(\bar{\mu}_1)\sim_k I_{s_1+s_2}^\delta(\Pi^{\pp 2^k})$ but the involved constant depending on $k$ is harmless.)
    This shows that the Fourier decay property holds for $n$, with constants $\epsilon_n:= \min\{\epsilon,\epsilon_{n - 1}\}$ and $\tau_n=\frac{\tau_{n-1}}{2^k}$.
    The necessary size of $k$ is determined by the application of Lemma~\ref{lemma3}, so it depends only on $\min\{s_{1},s_{2}\} > 0$.
    The proof of Theorem \ref{main} is complete.
\end{proof}

\section{A quantitative projection estimate}

For the proof of Theorem~\ref{second}, we first need to establish a quantitative combinatorial estimate on the size of projections of discretised sets.
Recall that for a set $X$ in $\R^d$, we write $\abs{X}_\delta$ to denote the covering number of $X$ at scale $\delta$.

\begin{definition}\label{def:deltaSSet}
    Given $s\geq 0$ and $K\geq 0$, a set $X \subset \R^d$ is a \emph{$(\delta,s,K)$-set} if it satisfies
    \begin{equation} \label{eq:delta-s-set}
        \forall x\in\R^d,\ \forall r\in[\delta,1],\quad
        \abs{X\cap B(x,r)}_\delta \leq K\cdot r^s\cdot\abs{X}_\delta.
    \end{equation}
    If the constant $K$ is universal we say that $X$ is a $(\delta,t)$-set.
\end{definition}
This definition differs from a related notion due to Katz and Tao, which is recalled as Definition \ref{def:KatzTaoSet} below. Note that the above definition closely resembles a Frostman condition.

The problem of estimating the dimension of projections of fractal sets has been extensively studied since the works of Marstrand~\cite{marstrand} and Kaufman~\cite{kaufman}. Ren and Wang~\cite{RenWang23} recently obtained the following optimal estimate for projections of discretised subsets of the plane.
Given $y \in \R$, we define the projection $\pi_y\colon\R^2\to\R$ by the formula
\begin{equation}\label{def:projections}
    \pi_y(x_1,x_2) = x_1 - y x_2. \end{equation}

\begin{thm}[Ren-Wang projection theorem {\cite[Theorem~4.1]{RenWang23}}]
    \label{thm:renwang}
    Fix parameters $s,t\in(0,1]$ and $\eta>0$.
    There exists $\epsilon=\epsilon(s,t,\eta)>0$ such that the following holds for all $\delta<\delta_0(s,t,\eta)$.\\
    If $X\subset[-1,1]^2$ is a $(\delta,2s,\delta^{-\epsilon})$ set and $Y\subset[0,1]$ a $(\delta,t,\delta^{-\epsilon})$ set, then the set
    \[
        Y_\eta = \left\{y\in Y\ :\ \abs{\pi_y(X)}_\delta \geq \delta^{-\min\{2s,s+\frac{t}{2},1\}+\eta}\right\}
    \]
    satisfies $\abs{Y_\eta}_\delta \geq (1-\delta^\epsilon)\abs{Y}_\delta$.
\end{thm}

\begin{remark}
    An estimate similar to (but weaker than) Theorem~\ref{thm:renwang} was previously obtained in \cite{OrponenShmerkin23}. For the purposes of establishing Theorem \ref{thm:rescy} below, and consequently also Theorem \ref{second}, the bounds from \cite{OrponenShmerkin23} would suffice (only the value of the constant $c$ in Theorem \ref{thm:rescy} would be different). We also note that the proof of \cite[Theorem~4.1]{RenWang23} builds upon~\cite{OrponenShmerkin23}, and in particular on the estimate given above as Proposition~\ref{prop:expansion}, which was already central in the proof of Theorem~\ref{main}.
\end{remark}

Our goal in this section is to derive from Theorem \ref{thm:renwang} a similar statement, but with weaker non-concentration assumptions on the sets $X$ and $Y$. (To be accurate, we will use the more technical "Furstenberg set version" of Theorem \ref{thm:renwang}, stated in Theorem \ref{thm:renWang23}.)
The lower bound on the projection's dimension will be of the form $s+ct$ for some absolute constant $c$; this is weaker than the conclusion of Theorem~\ref{thm:renwang}, but sufficient to obtain the desired exponential lower bound for the Fourier decay exponent of multiplicative convolutions.

\begin{thm}\label{thm:rescy}
    There exists a universal constant $c>0$ such that the following holds.
    (One may take $c=\frac{1}{24}$.)
    Fix $0 < t \leq s \le 3/4$.
    Let $A_i\subset [0,1]$ be $(\delta,s, \delta^{-ct})$-sets with $|A_i|_{\delta}\leq \delta^{-s}$. 
    Let $Y\subset [0,1]$ be a $(\delta,t,\delta^{-ct})$-set.
    Then there is $y\in Y$ such that
    \begin{equation*}
        |\pi_y(A_1\times A_2)|_{\delta} \ge \delta^{-s-ct}.
    \end{equation*}
\end{thm}

The deduction of the above result from Theorem~\ref{thm:renwang} goes in two steps. First, we use rescaling on $X$ and a combinatorial argument to weaken the non-concentration on $X$; then, we use rescaling on $Y$ and pigeonholing to obtain the desired estimate.

\subsection{Rescaling $X$}
In addition to the notion of $(\delta,s,K)$-set above, we shall also use the following variant.

\begin{definition}\label{def:KatzTaoSet}
    We say that $X$ is a \emph{Katz-Tao $(\delta,s,K)$-set} if
    \begin{equation} \label{eq:Katz-Tao-delta-s-set}
        |X\cap B_{r}|_{\delta} \le K\cdot \left(\frac{r}{\delta}\right)^s,\quad r\in [\delta,1].
    \end{equation}
\end{definition}

\begin{definition}
    If $\mathcal{P}$ is a family of dyadic $\delta$-cubes, we say that $\mathcal{P}$ is a (Katz-Tao) $(\delta,s,K)$-set if $\cup \mathcal{P}$ is a (Katz-Tao) $(\delta,s,K)$-set in the sense of Definitions \ref{def:deltaSSet} and \ref{def:KatzTaoSet}.
\end{definition}

\begin{remark}
    A $(\delta,s,K)$-set $X$ with $|X|_{\delta} \leq \delta^{-s}$ is also a Katz-Tao $(\delta,s, K)$-set, and conversely, a Katz-Tao $(\delta,s,K)$-set with $\abs{X}_\delta\geq\delta^{-s}$ is a $(\delta,s,K)$-set.
\end{remark}

For the first step of the proof of Theorem~\ref{thm:rescy}, we introduce some notation and terminology. Let $D,m \in \N$. We use $\mathcal{D}_{\Delta}(X)$ to denote the dyadic $\Delta$-cubes intersecting a set $X$. A set $X\subset [0,1]^d$ is \emph{$\{2^{-Dj}\}_{j = 1}^{m}$-uniform} if for all $1 \leq j \leq m$, there exists $R_j \in \N$ such that
\begin{displaymath} |X\cap Q|_{2^{-Dj}}=R_j, \qquad Q \in \mathcal{D}_{2^{-D(j - 1)}}(X). \end{displaymath}
Given a $\{2^{-Dj}\}_{j = 1}^{m}$-uniform set $X \subset [0,1]^d$, and writing $\delta := 2^{-Dm}$, we define the \emph{branching function} $f\colon[0,1]\to[0,d]$ by
\[
    f(u) = \frac{\log\abs{X}_{\delta^u}}{\log1/\delta}
\]
if $\delta^u=2^{-Dj}$ for some $j$, and interpolating linearly for other values of $u$. We refer the reader to \cite[\S 2.3]{OrponenShmerkin23} for the elementary properties of uniform subsets and branching functions. Note that we use here a slightly different normalisation of the branching function than in \cite[Definition 2.21]{OrponenShmerkin23}. Our branching function is $d$-Lipschitz on $[0,1]$.

Below, we shall use several times the following observation, see \cite[Lemma 3.6]{MR3919361}: Given $\epsilon>0$, any set $X\subset [0,1]^d$ contains a $\{2^{-Dj}\}_{j = 1}^{m}$-uniform subset $X'\subset X$ with $|X'|_\delta\ge \delta^{\epsilon}|X|_\delta$ provided that $\delta = 2^{-Dm}$, and both $D,m \in \N$ are sufficiently large in terms of $\epsilon>0$. This will allow us to assume that sets are $\{2^{-Dj}\}_{j = 1}^{m}$-uniform without loss of generality. If the choice of $D \geq 1$ is not important, we will abbreviate "$\{2^{-Dj}\}_{j = 1}^{m}$-uniform" to simply "uniform".

In the rest of this section, we hide small negative powers of $\delta$ in the notation $\lessapprox$, $\gtrapprox$, and $\approx$.
More precisely, we write $F(\delta) \lessapprox G(\delta)$ if for all $\epsilon>0$, one has $F(\delta)\leq \delta^{-\epsilon} G(\delta)$ provided $\delta$ is sufficiently small in terms of $\epsilon$. We can then recast the above observation as follows: a set $X\subset [0,1]^d$ contains a uniform subset $X'$ with $|X'|_\delta\gtrapprox \delta^{\epsilon}|X|_\delta$. (This involves a choice of $D$ depending on $\epsilon$, that is also hidden from the notation.)

Proposition \ref{prop:configurations} below uses some terminology and previous results from \cite{OS23,RenWang23,ShmerkinWang22b}, so we start by reviewing these. Below, $\mathcal{D}_{\delta}$ refers to the dyadic squares $p \subset [0,1]^{2}$ of side-length $\delta \in 2^{-\N}$, and $\mathcal{T}^{\delta}$ is the family of \emph{dyadic $\delta$-tubes}, see \cite[Definition 2.10]{OS23}. These are, by definition, sets of the form $T=\cup_{(a,b)\in p} \mathbf{D}(a,b)$, where $p\in\mathcal{D}_{\delta}$, and $\mathbf{D}$ is the usual point-line duality map $\mathbf{D}(a,b) := \{(x,y) \in \R^{2} : y = ax + b\}$. The reader may simply think that dyadic tubes are a convenient choice for $\delta$-discretising the family of all $\delta$-tubes in $\R^{2}$.

\begin{definition}[Nice configuration]
    Fix $\delta\in 2^{-\mathbb{N}}$, $s\in [0,1]$, $C>0$,  $M\in\mathbb{N}$. We say that a pair $(\mathcal{P},\mathcal{T}) \subset \mathcal{D}_{\delta} \times \mathcal{T}^{\delta}$ s a \emph{$(\delta,s,C,M)$-nice configuration} if for every $p\in\mathcal{P}$ there exists a $(\delta,s,C)$-set $\mathcal{T}(p) \subset\mathcal{T}$ with $ |\mathcal{T}(p)| = M$ and such that $T \cap p \neq\emptyset$ for all $T\in\mathcal{T}(p)$.
\end{definition}

It is relevant to find lower bounds for $|\mathcal{T}|$, where $(\mathcal{P},\mathcal{T})$ is a nice configuration, and $\mathcal{P}$ satisfies a non-concentration condition. This will also be the case in Proposition \ref{prop:configurations}. Lemma \ref{cor:ShWa} right below is a technical tool designed for this task: \eqref{form23} gives a lower bound for $|\mathcal{T}|$, provided that one can obtain lower bounds separately for $|\mathcal{T}_{\mathbf{p}}|$ for various restricted and rescaled configurations of the form $(S_{\mathbf{p}}(\mathcal{P} \cap \mathbf{p}),\mathcal{T}_{\mathbf{p}})$. Here and below, $S_{\mathbf{p}}$ is the rescaling affine map which sends $\mathbf{p} \in \mathcal{D}_{\Delta}$ to $[0,1)^{2}$.

\begin{lemma}[Corollary 4.1 in \cite{ShmerkinWang22b}]\label{cor:ShWa} Fix $N \geq 2$ and a sequence $\{\Delta_{j}\}_{j = 0}^{N} \subset 2^{-\N}$ with
    \begin{displaymath} 0 < \delta = \Delta_{N} < \Delta_{N - 1} < \ldots < \Delta_{1} < \Delta_{0} = 1. \end{displaymath}
    Let $(\mathcal{P}_{0},\mathcal{T}_{0}) \subset \mathcal{D}_{\delta} \times \mathcal{T}^{\delta}$ be a $(\delta,s,C,M)$-nice configuration. Then, there exists a set $\mathcal{P} \subset \mathcal{P}_{0}$ such that the following properties hold:
    \begin{itemize}
        \item[(i)] $|\mathcal{D}_{\Delta_{j}}(\mathcal{P})| \approx |\mathcal{D}_{\Delta_{j}}(\mathcal{P}_{0})|$ and $|\mathcal{P} \cap \mathbf{p}| \approx |\mathcal{P}_{0} \cap \mathbf{p}|$, $1 \leq j \leq N$, $\mathbf{p} \in \mathcal{D}_{\Delta_{j}}(\mathcal{P})$.
        \item[(ii)] For every $0 \leq j \leq N - 1$ and $\mathbf{p} \in \mathcal{D}_{\Delta_{j}}(\mathcal{P})$, there exist numbers
            \begin{displaymath} C_{\mathbf{p}} \approx C \quad \text{and} \quad M_{\mathbf{p}} \geq 1, \end{displaymath}
            and a family of tubes $\mathcal{T}_{\mathbf{p}} \subset \mathcal{T}^{\Delta_{j + 1}/\Delta_{j}}$, such that
            \begin{displaymath}
                (S_{\mathbf{p}}(\mathcal{P} \cap \mathbf{p}),\mathcal{T}_{\mathbf{p}}) \text{ is a  }(\Delta_{j + 1}/\Delta_{j},s,C_{\mathbf{p}},M_{\mathbf{p}})\text{-nice configuration}.
            \end{displaymath}

    \end{itemize}
    Furthermore, the families $\mathcal{T}_{\mathbf{p}}$ can be chosen so that if $\mathbf{p}_{j} \in \mathcal{D}_{\Delta_{j}}(\mathcal{P})$ for $0 \leq j \leq N - 1$, then
    \begin{equation}\label{form23} \frac{|\mathcal{T}_{0}|}{M} \gtrapprox \prod_{j = 0}^{N - 1} \frac{|\mathcal{T}_{\mathbf{p}_{j}}|}{M_{\mathbf{p}_{j}}}. \end{equation}
    All the constants implicit in the $\approx$ notation are allowed to depend on $N$.
\end{lemma}

Here is the main result in \cite{RenWang23}:

\begin{thm}[Theorem 4.1 in \cite{RenWang23}]\label{thm:renWang23}
    For every \(s \in (0,1]\), $t \in [0,2]$, and $\epsilon > 0$, there exist $\eta = \eta(\epsilon,s,t) > 0$ and $\delta_{0} = \delta_{0}(\epsilon,s,t) > 0$ such that the following holds for all $\delta \in (0,\delta_{0}]$.
    Let $(\mathcal{P},\mathcal{T})$ be a $(\delta,s,\delta^{-\eta},M)$-nice configuration, where $\mathcal{P} \subset \mathcal{D}_{\delta}$ is a non-empty $(\delta,t,\delta^{-\eta})$-set. Then,
    \begin{equation}\label{form24} \frac{|\mathcal{T}|}{M} \geq \delta^{-\min\{t,(s + t)/2,1\} + \epsilon}. \end{equation} \end{thm}

The proposition below has a similar flavour, but the hypothesis that $\mathcal{P}$ is a $(\delta,t,\delta^{-\eta})$-set has been replaced by the hypothesis that $\mathcal{P}$ is a Katz-Tao $(\delta,2s,K)$-set. In particular, note that this hypothesis gives no \emph{a priori} lower bound on the cardinality $|\mathcal{P}|$.

\begin{proposition} \label{prop:configurations}
    Fix $s \in (0,1)$, $0<t< 2\min\{s,1-s\} \leq 1$ and $C,K,M \ge 1$. Then the following holds for all $\delta \in 2^{-\N}$ with $0 < \delta \leq \delta_{0}(C,s,t)$. Let $\mathcal{P} \subset \mathcal{D}_{\delta}$ be a Katz-Tao $(\delta,2s,K)$-set, and let $\mathcal{T} \subset \mathcal{T}^{\delta}$ be a family of tubes such that $(\mathcal{P},\mathcal{T})$ is a $(\delta,t,C,M)$-nice configuration. Then,
    \begin{displaymath} \frac{|\mathcal{T}|}{M} \gtrapprox K^{-\frac{t}{4s(1-s)}}\cdot |\mathcal{P}|^{\frac{1}{2}+\frac{t}{4s}}, \end{displaymath}
    where the implicit constant may depend on $C$.
\end{proposition}

\begin{remark} One should think that the constant $K$ is very large, in fact a function of $\delta^{-1}$ in our application, whereas $C \geq 1$ is just an absolute constant (in fact, in our application it will depend on the implicit $\epsilon$). \end{remark}

\begin{proof}[Proof of Proposition \ref{prop:configurations}]
    Without loss of generality, $\mathcal{P}$ is $\{2^{-Di}\}_{i = 1}^{m}$-uniform, and $\delta = 2^{-Dm}$. Let $f$ be the branching function of $\mathcal{P}$. Given an interval $[a,b] \subset [0,1]$, and $s \geq 0$, we say that $(f,a,b)$ is \emph{$s$-superlinear if $f(u) - f(a) \geq s(u - a)$ for all $u \in [a,b]$.} Fix $\epsilon>0$. Using \cite[Lemma 5.21]{ShmerkinWang22}, decompose $[0,1]$ into intervals $[a_{j},a_{j+1}]$, $0 \leq j \leq N - 1$, such that $(f,a_{j},a_{j +1})$ is $s_{j}$-superlinear with $s_{j}$ increasing, $a_{j + 1} - a_{j} \geq \tau(\epsilon) > 0$, and
    \begin{equation} \label{eq:superlinear-decomposition}
        f(1)\ge  \sum_{j = 0}^{N - 1} s_i (a_{j +1}-a_{j}) \ge  f(1)-\epsilon.
    \end{equation}
    For the sake of clarity, we assume that $\epsilon=0$ in this equation; the interested reader can check that our argument also yields the limiting case as $\epsilon\downarrow 0$.

    Write $\Delta_{j} := \delta^{a_{j}} = 2^{-Dma_{j}}$ for $j \in \{0,\ldots,N\}$. Noting that $a_{0} = 0$ and $a_{N} = 1$,
    \begin{displaymath} \delta = 2^{-Dm} = \Delta_{N} < \ldots < \Delta_{0} = 1. \end{displaymath}
    We may therefore apply Corollary \ref{cor:ShWa} to the pair $(\mathcal{P}_{0},\mathcal{T}_{0}) := (\mathcal{P},\mathcal{T})$ and the scale sequence $\{\Delta_{j}\}_{j = 0}^{N}$. This leads to a refinement $\mathcal{P}' \subset \mathcal{P}$, and the families
    \begin{displaymath} \mathcal{T}_{\mathbf{p}} \subset \mathcal{T}^{\Delta_{j + 1}/\Delta_{j}}, \qquad \mathbf{p} \in \mathcal{D}_{\Delta_{j}}(\mathcal{P}'), \, 0 \leq j \leq N - 1. \end{displaymath}
    According to \eqref{form23},
    \begin{equation}\label{form25} \frac{|\mathcal{T}|}{M} \gtrapprox \prod_{j = 0}^{N - 1} \frac{|\mathcal{T}_{\mathbf{p}_{j}}|}{M_{\mathbf{p}_{j}}}, \qquad \mathbf{p}_{j} \in \mathcal{D}_{\Delta_{j}}(\mathcal{P}'). \end{equation}
    The plan of the remainder of the proof is to estimate each of the factors $|\mathcal{T}_{\mathbf{p}_{j}}|/M_{\mathbf{p}_{j}}$ separately, and finally multiply the results to obtain a lower bound for $|\mathcal{T}|/M$.

    Let $a = \min\{ a_{j} : s_{j} \ge t\}$ and $b=\min\{a_{j} : s_{j} \ge 2-t\}$, assuming that these exist. If $s_{j} < t$ for all $j$, we simply declare $a := 1$, and if $s_{j} < 2 - t$ for all $j$, we set $b := 1$.  Note that $t \leq 2 - t$, so $b \geq a$.

    Abbreviate $\delta_{j} := \Delta_{j + 1}/\Delta_{j}$ for $0 \leq j \leq N - 1$.
    Recall that $a_{j + 1} - a_{j} \geq \tau(\epsilon)$, so $\delta_{j} = \delta^{a_{j + 1} - a_{j}} \leq \delta^{\tau(\epsilon)}$.
    The information that $(f,a_{j},a_{j + 1})$ is $s_{j}$-superlinear translates to the property that
    \begin{displaymath}
        S_{\mathbf{p}}(\mathcal{P} \cap \mathbf{p}) \subset \mathcal{D}_{\delta_{j}}
    \end{displaymath}
    is a $(\delta_{j},s_{j},C)$-set for every $\mathbf{p} \in \mathcal{D}_{\Delta_{j}}(\mathcal{P})$, where $C > 0$ is a constant that depends only on $D$ in the definition of uniformity, which in turn has to be taken large enough in terms of $\epsilon$. See \cite[Lemma 8.3]{OS23}. Does the same remain true for the refinement $\mathcal{P}' \subset \mathcal{P}$? Almost: Lemma \ref{cor:ShWa}(i) guarantees that
    \begin{displaymath} |\mathcal{P}' \cap \mathbf{p}| \approx |\mathcal{P} \cap \mathbf{p}|, \qquad \mathbf{p} \in \mathcal{D}_{\Delta_{j}}(\mathcal{P}'). \end{displaymath}
    It then easily follows (or see \cite[Corollary 2.19]{OrponenShmerkin23}) from the $(\delta_{j},s_{j},C)$-set property of $S_{\mathbf{p}}(\mathcal{P} \cap \mathbf{p})$ that $S_{\mathbf{p}}(\mathcal{P}' \cap \mathbf{p})$ is a $(\delta_{j},s_{j},\approx 1)$-set. But since $\delta_{j} \leq \delta^{\tau(\epsilon)}$, logarithmic constants in $\delta$ are also logarithmic in $\delta_{j}$.
    In particular, for any given $\eta > 0$, and provided that $\delta > 0$ is sufficiently small, the sets $S_{\mathbf{p}}(\mathcal{P}' \cap \mathbf{p})$ are $(\delta_{j},s_{j},\delta_{j}^{-\eta})$-sets for $\mathbf{p} \in \mathcal{D}_{\Delta_{j}}(\mathcal{P}')$.

    We are then in a position to apply Theorem \ref{thm:renWang23} to the $(\delta_{j},s_{j},C_{\mathbf{p}_{j}},M_{\mathbf{p}_{j}})$-nice configurations $(S_{\mathbf{p}_{j}}(\mathcal{P}' \cap \mathbf{p}_{j}),\mathcal{T}_{\mathbf{p}_{j}})$, where $\mathbf{p}_{j} \in \mathcal{D}_{\Delta_{j}}(\mathcal{P}')$ is arbitrary, and $C_{\mathbf{p}_{j}} \approx C$ by Lemma \ref{cor:ShWa}(ii). In particular, $C_{\mathbf{p}_{j}} \leq \delta_{j}^{-\eta}$ for $\delta > 0$ sufficiently small. Taking the "$\min$" in \eqref{form24} into account, the conclusion of Theorem \ref{thm:renWang23} is that
    \begin{align*}
         & \frac{|\mathcal{T}_{\mathbf{p}_{j}}|}{M_{\mathbf{p}_{j}}} \gtrapprox \delta_{j}^{-s_{j}} = \delta^{-s_j(a_{j+1}-a_{j})},                      & a_j\in [0,a], \\
         & \frac{|\mathcal{T}_{\mathbf{p}_{j}}|}{M_{\mathbf{p}_{j}}} \gtrapprox \delta_{j}^{(s_{j} + t)/2} = \delta^{-\frac{s_{j}+t}{2}(a_{j+1}-a_{j})}, & a_j\in [a,b], \\
         & \frac{|\mathcal{T}_{\mathbf{p}_{j}}|}{M_{\mathbf{p}_{j}}} \gtrapprox \delta_{j}^{-1} = \delta^{-(a_{j+1}-a_j)},                               & a_j\in [b,1].\end{align*}
    Plugging this information into the lower bound \eqref{form25}, we find
    \begin{equation*}
        \log_{1/\delta} (|\mathcal{T}|/M) \ge  f(a)+ \frac{f(b)-f(a)}{2} + \frac{(b-a)t}{2} + 1-b.
    \end{equation*}
    We get
    \begin{equation} \label{eq:estimate-1}
        \begin{split}
            \log_{1/\delta}(|\mathcal{T}|/M) &\ge \frac{f(b)+f(a)}{2} +\frac{(b-a)t}{2} + 1-b \\
            & =  \frac{f(1)+f(a)}{2} +\frac{(1-a)t}{2}  + (1-b)(1-\frac{t}{2}) - \frac{f(1)-f(b)}{2}.
        \end{split}
    \end{equation}
    Write $K=:\delta^{-u}$. The assumption that $\mathcal{P}$ is a Katz-Tao $(\delta,2s,K)$-set translates to
    $f(1)-f(a) \le 2s(1-a)+u$. Since also $t\le 2s$ by assumption, we have
    \begin{equation} \label{eq:estimate-2}
        f(a)+(1-a)t \ge f(a)+\frac{t(f(1)-f(a)-u)}{2s}  \ge \frac{t f(1)}{2s}-\frac{tu}{2s}.
    \end{equation}

    We assume that $b < 1$. If $b = 1$, some of the terms in \eqref{eq:estimate-1} vanish, and a stronger form of \eqref{form22} already follows from \eqref{eq:estimate-2}. Write $\sigma :=\frac{f(1)-f(b)}{1-b}\in [2-t,2] \subset (2s,2]$.
    Using the Katz-Tao assumption on $\mathcal{P}$ again, we have
    \[
        \sigma(1-b) = f(1)-f(b) \le u + 2s(1-b) \Longrightarrow 1-b \le \frac{u}{\sigma-2s},
    \]
    and therefore
    \begin{equation} \label{eq:estimate-3}
        (1-b)(1-\frac{t}{2}) - \frac{f(1)-f(b)}{2}
        = \frac{1}{2}(1-b)(2-t-\sigma)
        \ge -\frac{u(t+\sigma-2)}{2(\sigma-2s)} \ge -\frac{tu}{4(1-s)},
    \end{equation}
    using that $\sigma \in [2(1-s),2]$ and the assumption $2s<2-t$ for the last inequality.

    Combining \eqref{eq:estimate-1}, \eqref{eq:estimate-2} and \eqref{eq:estimate-3}, we conclude that
    \begin{equation}\label{form22}
        \log_{1/\delta} (|\mathcal{T}|/M) \ge (1+\frac{t}{2s})\frac{f(1)}{2} - \frac{tu}{4s(1-s)},
    \end{equation}
    as claimed. (To be more precise, we should add a $-O_{s,t}(\epsilon)$ in the right-hand side of \eqref{form22}, to account from the $\varepsilon$ in \eqref{eq:superlinear-decomposition}.)
\end{proof}

We will only use the conclusion of Proposition \ref{prop:configurations} via the following corollary concerning projections:

\begin{cor} \label{lem:projections}
    Fix $s \in (0,1)$, $0<t< 2\min\{s,1-s\}$ and $C,K\ge 1$. Let $X\subset [0,1]^2$ be a Katz-Tao $(\delta,2s,K)$-set, and let $Y\subset [1,2]$ be a $\delta$-separated $(\delta,t,C)$-set.
    Then
    \[
        |\pi_{y} (X)|_{\delta} \gtrapprox K^{-\frac{t}{4s(1-s)}}\cdot |X|_{\delta}^{\frac{1}{2}+\frac{t}{4s}},
    \]
    for at least $|Y|/2$ values of $y\in Y$.
\end{cor}

\begin{proof} Given Proposition \ref{prop:configurations}, the proof of Corollary \ref{lem:projections} is exactly the same as the proof of \cite[Corollary 6.1]{OrponenShmerkin23} starting from \cite[Theorem 5.35]{OrponenShmerkin23}. We only sketch the idea. It is enough to establish the existence of one $y\in Y$ satisfying the desired conclusion; we can then apply this case to the exceptional set of $y\in Y$ that fail it. This is because any subset of $Y$ cardinality $\ge |Y|/2$ satisfies the same assumptions, up to a change in the constants.

    Write $\mathcal{P} := \mathcal{D}_{\delta}(X)$. Then $\mathcal{P}$ is a Katz-Tao $(\delta,2s,O(K))$-set. For each $y \in Y_{\delta}$ (a $\delta$-net inside $Y$), cover $\pi_{y}(\mathcal{P})$ by a family $\mathcal{T}_{y} \subset \mathcal{T}^{\delta}$ such that $|\mathcal{T}_{y}| \lesssim |\pi_{y}(\mathcal{P})|_{\delta} \sim |\pi_{y}(X)|_{\delta}$, and the slopes of the tubes in $\mathcal{T}_{y}$ are roughly parallel to the line $\pi_{y}^{-1}\{0\}$. Set
    \begin{displaymath} \mathcal{T} := \bigcup_{y \in Y_{\delta}} \mathcal{T}_{y}. \end{displaymath}
    Then $(\mathcal{P},\mathcal{T})$ is a $(\delta,t,C,M)$-nice configuration with $M = |Y_{\delta}|$. We may now infer from Proposition \ref{prop:configurations} that
    \begin{displaymath} \frac{|\mathcal{T}|}{|Y_{\delta}|} \gtrapprox K^{-\frac{t}{4s(1-s)}}\cdot |X|_{\delta}^{\frac{1}{2}+\frac{t}{4s}}. \end{displaymath}
    This implies the claim, taking into account that $|\mathcal{T}| \lesssim |Y_{\delta}| \cdot \max \{|\pi_{y}(X)|_{\delta} : y \in Y_{\delta}\}$. \end{proof}

\subsection{Rescaling $Y$}

We now derive Theorem~\ref{thm:rescy} from Corollary~\ref{lem:projections} by a rescaling argument on the set $Y$ of projection directions.

\begin{proof}[Proof of Theorem~\ref{thm:rescy}]
    By passing to suitable $\delta$-separated subsets, we may assume that $A_1, A_2$ and $Y$ are $\delta$-separated. Note that the assumptions imply that $A_1$ and $A_2$ are Katz-Tao $(\delta,s,\delta^{-ct})$-sets.
    Let $I\subset[0,1]$ be an interval of minimal length $\rho$ such that
    \[
        \abs{Y\cap I}_\delta \geq \rho^{\frac{t}{2}}\abs{Y}_\delta.
    \]
    Since $Y$ is a $(\delta,t,\delta^{ct})$-set,
    \begin{equation} \label{eq:rho-bound}
        1\ge \rho \ge \delta^{2c}.
    \end{equation}
    Write $I=[y_0,y_0+\rho]$ and let $X$ be the image of $A_1\times A_2$ under the map
    \[
        u_{y_0}\colon (a_1,a_2) \mapsto (a_1-y_0a_2,a_2).
    \]
    The projections of $A_1\times A_2$ with slopes in $Y\cap I$ are precisely the projections of $X$ with slopes in the subset
    \[
        Z = \{ y-y_0 : y\in Y\cap I\} \subset [0,\rho].
    \]
    In other words
    \begin{equation}\label{form32} \pi_{y - y_{0}}(X) = \pi_{y}(A_{1} \times A_{2}), \end{equation}
    where $\pi_{y}(x_{1},x_{2}) = x_{1} - yx_{2}$ (recall \eqref{def:projections}). Given $\lambda>0$, define $S_\lambda(x)=\lambda x$.
    By our choice of $\rho$, the set $Z' := S_{\rho^{-1}}(Z)$ is a $(\rho^{-1}\delta,t/2)$-set.

    Decompose $X$ into rectangular blocks
    \begin{equation*}
        X_j = X\cap \big([j\rho,(j+1)\rho)\times [0,1]\big).
    \end{equation*}
    Note that for $z\in[0,\rho]$, the sets $\{\pi_z(X_j)\}_j$ have bounded overlap for any $z\in Z$, and so
    \begin{equation} \label{eq:disjoint}
        \abs{\pi_y(A_1\times A_2)}_\delta =  |\pi_{y - y_{0}} X|_{\delta} \gtrsim \sum_j |\pi_{y - y_{0}} X_j|_{\delta}, \qquad y \in Y \cap I.
    \end{equation}
    Let $\mu_X$ be the uniform probability measure on $X$, and $\mu=(\pi_0)_{\#}\mu_X$ its projection to the first coordinate.
    For any ball $B_r\subset [0,1]$ of radius $r\in [\delta,1]$,
    \begin{align*}
        \mu(B_r) & = \frac{1}{\abs{A_1\times A_2}} \abs{\{(a_1,a_2)\in X\ :\ a_1-y_0a_2\in B_r\}}            \\
                 & = \frac{1}{\abs{A_1}\abs{A_2}} \sum_{a_2\in A_2} \abs{\{ a_1\ : \ a_1\in (B_r+y_0a_2) \}} \\
                 & \leq \delta^{-ct} \cdot r^{s}.
    \end{align*}
    By the dyadic pigeonhole principle, we may choose $\lambda>0$ and a set $H'\subset [0,1]$ such that $\mu(H')\gtrapprox 1$ and $\mu(I)\sim \lambda$ for all $\delta$-intervals $I$ intersecting $H'$. Let $H$ be a  uniform subset of $H'$ with $|H'|_{\delta}\gtrapprox |H|_{\delta}$.
    Then $H$ is a $(\delta,s,\approx\delta^{-ct})$-set with $\mu(H) \gtrapprox 1$.

    From now on, we consider only the set $J_{H} = \{j : \pi_0(X_j)\cap H\neq\emptyset\}$. Note that, since $H$ is a $(\delta,s,\lessapprox\delta^{-ct})$-set, we have
    \begin{equation} \label{eq:lower-N}
        |J_{H}|= |H|_{\rho} \gtrapprox \delta^{ct} \rho^{-s}.
    \end{equation}
    Moreover, for all $j \in J_{H}$, writing $I_j=[j\rho,(j+1)\rho]$,
    \[
        \frac{\abs{X_j}_\delta}{\abs{X}_\delta} \sim \mu_{X}(X_{j}) = \mu(I_j)
        \geq \mu(I_j\cap H)
        \approx \lambda \abs{I_j\cap H}_\delta
        = \lambda \abs{H}_\delta \abs{H}_\rho^{-1}.
    \]
    Since $\lambda \abs{H}_\delta\sim \mu(H)\gtrapprox 1$, we find
    \begin{equation}\label{eq:size-X'}
        \abs{X_j}_\delta \gtrapprox \abs{H}_\rho^{-1}\abs{X}_\delta, \qquad j\in J_{H}.
    \end{equation}
    By \eqref{eq:disjoint},
    \begin{equation}  \label{eq:lower-proj-coarse}
        \frac{1}{|Y \cap I|} \sum_{y \in Y \cap I} |\pi_y(A_1\times A_2)|_{\delta} = \frac{1}{|Z|} \sum_{z \in Z} \abs{\pi_z X}_\delta \gtrsim \sum_{j \in J_{H}} \frac{1}{|Z|} \sum_{z \in Z} |\pi_{z}X_{j}|_{\delta}.
    \end{equation}
    We now seek to find a lower bound for the inner sum. Fix $j \in J_{H}$, and write $X' := X_{j}$. Without loss of generality, assume $X'\subset [0,\rho]\times [0,1]$.

    Writing
    \begin{equation*}
        \pi_z(x_1,x_2) = S_{\rho}\left( \rho^{-1} x_1+(\rho^{-1}z) x_2\right),
    \end{equation*}
    we see that
    \begin{equation*}
        \pi_z X' = S_{\rho}\pi_{\rho^{-1}z} X'',\quad\text{where }X''= L_{\rho}(X'),\, L_{\rho}:(x_1,x_2)\mapsto (\rho^{-1} x_1,x_2).
    \end{equation*}
    In particular,
    \begin{displaymath}
        |\pi_z X'|_{\delta} = |\pi_{\rho^{-1}z} X''|_{\rho^{-1}\delta}, \qquad z \in Z.
    \end{displaymath}
    Consequently, recalling that $Z' = S_{\rho^{-1}}(Z)$,
    \begin{equation}\label{eq:X'-rescaling} \frac{1}{|Z|} \sum_{z \in Z} |\pi_{z}X'|_{\delta} = \frac{1}{|Z'|} \sum_{z' \in Z'} |\pi_{z'}X''|_{\rho^{-1}\delta}. \end{equation}
    Recall (from below \eqref{form32}) that $Z'$ is a $(\rho^{-1}\delta,t/2)$-set. This will enable us to find a good lower bound for the right hand side, by applying Corollary \ref{lem:projections} to $X''$ and $Z'$.
    We start by estimating $|X''|_{\rho^{-1}\delta}$.
    Since $X''=L_{\rho} X'$, we know that $|X''|_{\rho^{-1}\delta}$ is the number of $(\delta\times \rho^{-1}\delta)$-mesh rectangles intersecting $X'$.
    Any such rectangle $R$ satisfies
    \begin{equation*}
        |X'\cap R|_{\delta}
        \lesssim \delta^{-ct} \rho^{-s}.
    \end{equation*}
    Indeed,
    \(
    |X'\cap R|_{\delta}
    \leq \abs{X \cap R}_\delta
    \leq \abs{A_2 \cap B_{\rho^{-1}\delta}}_\delta
    \)
    and, since $A_2$ is a Katz-Tao $(\delta,s,\delta^{-ct})$-set, one has $|A_2\cap B_{\rho^{-1}\delta}|_{\delta} \le \delta^{-ct} \rho^{-s}$.
    We deduce that
    \begin{equation} \label{eq:X''-size}
        |X''|_{\rho^{-1}\delta}  \gtrsim \delta^{ct} \cdot \rho^{s} \cdot |X'|_{\delta}
        \overset{\eqref{eq:size-X'}}{\gtrapprox} \delta^{ct}\cdot \rho^{s} \cdot |H|_{\rho}^{-1} \cdot |X|_{\delta} .
    \end{equation}

    We claim that in addition $X''$ is a Katz-Tao $(\rho^{-1}\delta,2s,\delta^{-4cs})$-set.
    To see this, fix $r\in[\rho^{-1}\delta,1]$ and consider a square $Q_r=I_r\times J_r$ of size $r$ in $\R^2$.
    If $(x_1,x_2)\in X''\cap Q$ is defined up to an error $\rho^{-1}\delta$, then $x_2\in A_2\cap J_r$ is defined up to $\rho^{-1}\delta$, and then
    \(
    \rho x_1 \in (A_1-y_0x_2)\cap (\rho I_r)
    \)
    is defined up to $\delta$.
    Therefore,
    \begin{align*}
        \abs{X''\cap Q_r}_{\rho^{-1}\delta}
         & \lesssim \abs{A_1\cap(\rho I_r)}_\delta \cdot \abs{A_2\cap J_r}_{\rho^{-1}\delta} \\
         & \leq \abs{A_1\cap(\rho I_r)}_\delta \cdot \abs{A_2\cap J_r}_{\delta}.
    \end{align*}
    Since $A_1$ and $A_2$ are Katz-Tao $(\delta,s,\delta^{-ct})$-sets, the above inequality implies
    \begin{align*}
        \abs{X''\cap Q_r}_{\rho^{-1}\delta}
         & \leq \delta^{-ct} \rho^{s} r^{s} \delta^{-s} \cdot \delta^{-ct} r^{s} \delta^{-s} \\
         & \leq \delta^{-2ct} \rho^{-s} r^{2s} (\rho^{-1}\delta)^{-2s}                       \\
         & \leq \delta^{-4cs} \left(\frac{r}{\rho^{-1}\delta}\right)^{2s},
    \end{align*}
    where we used the assumption $t\leq s$ and \eqref{eq:rho-bound} for the last inequality.

    We are ready to apply Corollary \ref{lem:projections} to $X''$ and $Z'$, with parameters $\rho^{-1}\delta$ in place of $\delta$, $t/2$ in place of $t$, and $\delta^{-4cs}$ in place of $K$.
    Recall also that $t\le s \le 3/4$, so that in particular, $t/2\leq s/2 \leq s$ and $t/2\leq 1/2 \leq 2(1-s)$. The conclusion is that
    \begin{align*}
        \frac{1}{|Z'|} \sum_{z' \in Z'} |\pi_{z'} X''|_{\rho^{-1}\delta} & \gtrapprox \delta^{\frac{ct}{2(1-s)}} \cdot |X''|_{\rho^{-1}\delta}^{\frac{1}{2}+\frac{t}{8s}}                                                                                 \\
                                                                         & \ge \delta^{2ct} \cdot |X''|_{\rho^{-1}\delta}^{\frac{1}{2}+\frac{t}{8s}}                                                                                                      \\
                                                                         & \overset{\eqref{eq:X''-size}}{\gtrapprox} \delta^{2ct} \cdot (\delta^{ct}\rho^{s}\abs{H}_\rho^{-1})^{\frac{1}{2}+\frac{t}{8s}} \cdot \abs{X}_\delta^{\frac{1}{2}+\frac{t}{8s}} \\
                                                                         & \gtrapprox \delta^{\frac{31ct}{8}}\cdot  \big(\rho^{s}|H|_{\rho}^{-1}\big)^{\frac{1}{2}+\frac{t}{8s}} \cdot  \delta^{-s -\frac{t}{4}},
    \end{align*}
    using that $|X|_{\delta}\sim \abs{A_1\times A_2} \ge \delta^{2ct-2s}$ and $t\leq s$ in the last line. This estimate is valid for every $X'' := X_{j}'' := L_{\rho}(X_{j})$, where $j \in J_{H}$.

    Recalling \eqref{eq:lower-proj-coarse}-\eqref{eq:X'-rescaling}, \eqref{eq:lower-N} and \eqref{eq:rho-bound}, we conclude that
    \begin{align*}
        \tfrac{1}{|Y \cap I|} \sum_{y \in Y \cap I} \abs{\pi_y(A_1\times A_2)}_\delta & \gtrsim \sum_{j \in J_{H}} \tfrac{1}{|Z'|} \sum_{z' \in Z'} |\pi_{z'}X_{j}''|_{\rho^{-1}\delta}                                                              \\
                                                                                      & \gtrapprox |H|_{\rho} \cdot \delta^{\frac{31ct}{8}}\cdot  \big(\rho^{s}|H|_{\rho}^{-1}\big)^{\frac{1}{2}+\frac{t}{8s}} \cdot  \delta^{-s -\frac{t}{4}}       \\
                                                                                      & \gtrapprox \delta^{\frac{31}{8}ct} \cdot \big(\rho^{s}\big)^{\frac{1}{2}+\frac{t}{8s}} \abs{H}_\rho^{\frac{1}{2}-\frac{t}{8s}} \cdot \delta^{-s-\frac{t}{4}} \\
                                                                                      & \overset{\eqref{eq:lower-N}}{\gtrapprox} \delta^{\frac{35}{8}ct} \cdot \big(\rho^{s}\big)^{\frac{1}{2}+\frac{t}{8s}}
        \big(\rho^{-s}\big)^{\frac{1}{2}-\frac{t}{8s}}  \cdot \delta^{-s-\frac{t}{4}}                                                                                                                                                                \\
                                                                                      & \overset{\eqref{eq:rho-bound}}{\gtrapprox} \delta^{\frac{40}{8} ct}\cdot \delta^{-s-\frac{t}{4}}.
    \end{align*}
    Taking $c=1/24$ we get the desired conclusion.
\end{proof}

\section{An exponential lower bound for the exponent}

With Theorem~\ref{thm:rescy} at hand, the proof of Theorem~\ref{second} follows the same general scheme as that of Theorem~\ref{main}.
One first proves a flattening statement in the spirit of Lemma~\ref{lemma3}, and then applies it iteratively to obtain a measure $\Pi_n$ with Fourier decay.

Finally, Lemma~\ref{order} allows one to compare the Fourier transform of $\mu_1\ff\dots\ff\mu_n$ with that of $\Pi_n$ to get the desired inequality.

\subsection{Flattening lemma}

The lemma below is similar to Lemma~\ref{lemma3}, but uses only one additive convolution. The bound obtained on the energy of $\Pi$ should be understood as $\dim(AB-AB)\geq \dim A + \frac{1}{C}\dim B$ in an energy sense.

For the proof, we use an argument originating in the work of Bourgain, Glibichuk and Konyagin~\cite{bgk}, based on the Balog-Szemerédi-Gowers lemma, see also Li~\cite{Li21} for a setting closer to the one studied here.

\begin{lemma}[Energy flattening]
    \label{flattening}
    There exists a universal constant $C>0$ such that for all $\epsilon>0$ the following holds for all $\delta > 0$ small enough.

    Let $\mu$ and $\nu$ be two probability measures on $[1,2]$, 
    and set
    \[
        \Pi = (\mu\ff\nu) \mm (\mu\ff\nu).
    \]
    Let $0< t \leq s\leq 3/4$ and assume $I_s^\delta(\mu)\leq\delta^{-\epsilon}$ and $I_t^\delta(\nu)\leq\delta^{-\epsilon}$.
    Then
    \[
        I_{s+t/C}^\delta(\Pi) \leq \delta^{-C\epsilon/t}.
    \]
\end{lemma}

The key step in the proof of Lemma~\ref{flattening} is the following single-scale result.
\begin{lemma}
    \label{flattening-keystep}
    Under the assumptions of Lemma~\ref{flattening}, for all $\rho\in[\delta,\delta^{\epsilon/t}]$ one has
    \[
        \norm{\mu_{\rho}}_2^2 \ge\rho^{-1+s+t/C} \implies \norm{\Pi_\rho}_2^2 \leq \rho^{\tau}\norm{\mu_\rho}_2^2,
    \]
    if $C$ is a sufficiently large constant.
\end{lemma}

We postpone the proof of Lemma~\ref{flattening-keystep} to the next section, and first show how it implies Lemma~\ref{flattening}.
\begin{proof}[Proof of Lemma \ref{flattening}, assuming Lemma \ref{flattening-keystep}]

    Fix $\rho\in[\delta,\delta^{\epsilon/\tau}]$. Combining Lemma \ref{flattening-keystep} and the trivial bound $\norm{\Pi_\rho}_2^2 \leq \norm{\mu_\rho}_2^2$ for the case $\norm{\mu_\rho}_2^2\leq \rho^{-1+s+\tau}$, we get
    \[
        \rho^{1-s-\tau}\norm{\Pi_\rho}_2^2 \leq 1 + \rho^{1-s}\norm{\mu_\rho}_2^2.
    \]
    For $\rho \in [\delta^{\epsilon/\tau},1]$, we use the simple bound
    \[
        \rho^{1-s-\tau}\norm{\Pi_\rho}_2^2 \leq \delta^{-\epsilon/\tau}.
    \]
    Thus,
    \begin{align*}
        I_{s+\tau}^\delta(\Pi)
         & \lesssim \sum_{\delta\leq \rho=2^{-k}\leq 1} \rho^{1-s-\tau}\norm{\Pi_\rho}_2^2                                                    \\
         & \lesssim (\log 1/\delta)\delta^{-\epsilon/\tau} + (\log1/\delta) + \sum_{\delta\leq \rho=2^{-k}\leq1}\rho^{1-s}\norm{\mu_\rho}_2^2 \\
         & \lesssim (\log 1/\delta)\delta^{-\epsilon/\tau} + I_{s}^\delta(\mu).
    \end{align*}

\end{proof}

\subsection{Proof of Lemma \ref{flattening-keystep}}

We begin with an elementary lemma that relates measures with small energy and the non-concentration property for discretised sets.
\begin{lemma}[Energy and non-concentrated subsets]
    \label{lem:energy}
    Let $\nu$ be a probability measure on $[0,1]$ such that $I_s^\rho(\nu)<\rho^{-2\tau}$.
    Any union of $\rho$-intervals $A\subset[0,1]$ such that $\nu_\rho(A)\geq\rho^\tau$ contains a $(\rho,s,\rho^{-6\tau})$-set $A_1$ such that $\nu(A_1)\geq\rho^{2\tau}$.
\end{lemma}
\begin{proof}
    Indeed, from Lemma~\ref{frostman-energy}, there exists a set $E$ with measure $\nu(E)\leq (\log1/\rho)\rho^{2\tau}$ such that for all $x\not\in E$, for all $r\in[\rho,1]$,
    \[
        \nu(B(x,r)) \leq \rho^{-4\tau} r^s.
    \]
    The set $A'=A\setminus E$ satisfies $\nu(A')\geq\rho^\tau/2$.
    Taking dyadic level sets $2^{-k}$, for $k=0,\dots,\log1/\rho$, we may further restrict $A'$ to a subset $A_1$ such that $\nu(A_1)\geq\rho^{2\tau}$ and the density of $\nu_\rho$ is essentially constant on $A_1$:
    \[
        \exists D>0:\forall x\in A_1,\quad D \leq \nu_\rho(x) \leq 2D.
    \]
    Then, for all $x$ in $A_1$ and all $r\geq\rho$,
    \[
        \frac{\abs{A_1\cap B(x,r)}_\rho}{\abs{A_1}_\rho}
        \asymp \frac{\nu_\rho(A_1\cap B(x,r))}{\nu_\rho(A_1)}
        \lesssim \rho^{-2\tau} \nu_\rho(B(x,r))
        \leq \rho^{-6\tau} r^s.
    \]
    Enlarging the implicit constant, one readily checks that this estimate holds in fact for all $x$ in $\R$.
    (Indeed, if $B(x,r)\cap A_1=\varnothing$, there is nothing to prove, and if $B(x,r),\cap A_1\ni x_0$, then write $B(x,r)\subset B(x_0,2r)$.)
    This shows that $A_1$ is a $(\rho,s,\rho^{-6\tau})$-set.
\end{proof}

\begin{proof}[Proof of Lemma \ref{flattening-keystep}]
    Let $\tau = t/C$, for some large enough universal constant $C$. Assume
    \begin{equation}\label{eq:mularge}
        \norm{\mu_\rho}_2^2\geq \rho^{-1+s+\tau}.
    \end{equation}
    We wish to show the flattening estimate
    \begin{equation}\label{flatteningPi}
        \norm{\Pi_\rho}_2^2 \leq \rho^{\tau}\norm{\mu_\rho}_2^2.
    \end{equation}
    Assume, for the sake of contradiction, that \eqref{flatteningPi} fails. We discretise the measure $\mu$ at scale $\rho$ using dyadic level sets, and then use the quantitative projection estimate obtained in the previous section to derive a contradiction.

    \bigskip
    \noindent\underline{Step~1: Choosing large level sets for $\mu$}\\
    It is easy to check that there exist sets $A_i \subset [1,2]$, $i \geq 0$, which can be expressed as unions of $\rho$-intervals, such that $A_i$ is empty for $i \gtrsim \log\frac{1}{\rho}$, and
    \begin{equation}
        \label{eq:mullAi}
        \mu_\rho
        \lesssim \sum_{i \geq 0} 2^i \mathbf{1}_{A_i}
        \lesssim \mu_{3\rho} + 1.
    \end{equation}
    We have
    \[
        (\mu_\rho \ff \nu) \mm (\mu_\rho\ff\nu)
        = \iint_{\R \times \R} (\mu_\rho \ff \delta_{x}) \mm (\mu_\rho \ff \delta_y) \,d\nu(x) \,d\nu(y).
    \]
    Therefore, from the left inequality in \eqref{eq:mullAi},
    \[
        (\mu_\rho \ff\nu) \mm (\mu_\rho \ff \nu)
        \lesssim \sum_{i,j \geq 0} 2^{i + j} \iint (\mathbf{1}_{A_i} \ff \delta_{x}) \mm  (\mathbf{1}_{A_j} \ff \delta_y) \,d\nu(x) \,d\nu(y).
    \]
    Observe that $\mathbf{1}_{A_i} \ff \delta_{x} = \abs{x}^{-1}\mathbf{1}_{xA_i}$ and $\mathbf{1}_{A_j} \ff \delta_y = \abs{y}^{-1}\mathbf{1}_{yA_j}$. Since $\spt \nu \subset [1,2]$, the factors $|x|^{-1},|y|^{-1}$ lie in $[\tfrac{1}{2},1]$, so
    \[
        (\mu_\rho \ff \nu) \mm (\mu_\rho \ff \nu)
        \lesssim \sum_{i,j \geq 0} 2^{i + j} \iint  \mathbf{1}_{xA_i} \mm \mathbf{1}_{yA_j} \,d\nu(x) \,d\nu(y).
    \]
    The triangle inequality shows that
    \[
        \sum_{i,j \geq 0} 2^{i + j} \iint  \norm{\mathbf{1}_{xA_i} \mm \mathbf{1}_{yA_j}}_2 \,d\nu(x) \,d\nu(y)
        \gtrsim \norm{\Pi_\rho}_2.
    \]
    There are at most $O\bigl( (\log \frac{1}{\rho})^2 \bigr)\lessapprox 1$  terms in this sum (here and below, we use $\lessapprox$ to hide logarithmic functions of $\rho$). Hence, there exist $i,j \geq 0$ such that
    \begin{equation}
        \label{eq:defiandj}
        2^{i + j} \iint \norm{\mathbf{1}_{xA_i} \mm \mathbf{1}_{yA_j}}_2 \,d\nu(x) \,d\nu(y)
        \gtrapprox \norm{\Pi_\rho}_2.
    \end{equation}
    From now on we fix such $i$ and $j$.
    Note that, by \eqref{eq:mullAi} and Young's inequality, for all $x,y$,
    \[
        2^{i+j}\norm{\mathbf{1}_{xA_i} \mm \mathbf{1}_{yA_j}}_2  \lesssim \norm{\mu_{3\rho}}_1\norm{\mu_{3\rho}}_2 \lesssim \norm{\mu_\rho}_2.
    \]
    In particular, under our counter-assumption $\norm{\Pi_\rho}_2^2\geq\rho^{\tau}\norm{\mu_\rho}_2^2$, one finds that (if $\rho$ is small enough) there exists $x$ such that the set
    \begin{equation} \label{eq:def-Y0}
        Y_0 = \bigl\{\, y  :  2^{i+j}\norm{\mathbf{1}_{xA_i} \mm \mathbf{1}_{yA_j}}_2  \geq \rho^{\tau} \norm{\mu_\rho}_2 \,\bigr\}
    \end{equation}
    satisfies
    \begin{equation} \label{eq:nu-Y0}
        \nu (Y_0) \gtrapprox  \rho^{\frac{\tau}{2}}-\rho^{\tau}
        \ge \rho^{\frac{\tau}{2}}.
    \end{equation}
    We fix this value of $x$ for the rest of the proof.

    Observe that, by \eqref{eq:mullAi},
    \begin{equation}\label{above_ai}
        1\gtrsim 2^i\norm{\mathbf{1}_{xA_i}}_1 \sim 2^i\abs{A_i}
        \quad\mbox{and}\quad
        \norm{\mu_\rho}_2 \gtrsim 2^i\norm{\mathbf{1}_{xA_i}}_2 \sim 2^i \abs{A_i}^{1/2},
    \end{equation}
    where $|\cdot|$ denotes the Lebesgue measure, and similarly
    \begin{equation}\label{above_aj}
        1\gtrsim 2^j\norm{\mathbf{1}_{yA_j}}_1 \sim 2^j\abs{A_j}
        \quad\mbox{and}\quad
        \norm{\mu_\rho}_2 \gtrsim 2^j\norm{\mathbf{1}_{yA_j}}_2 \sim 2^j \abs{A_j}^{1/2}.
    \end{equation}
    Using also the definition of $Y_0$ and Young's inequality, we find
    \begin{displaymath}
        \rho^\tau \cdot \norm{\mu_\rho}_2 \leq 2^{i+j} \norm{\mathbf{1}_{x A_i} \mm \mathbf{1}_{yA_j}}_2 \lesssim 2^i\abs{A_i} \cdot 2^j \abs{A_j}^{\frac{1}{2}}, \quad y\in Y_0.
    \end{displaymath}
    Using \eqref{above_ai} and \eqref{above_aj} above, as well as the symmetry between $i$ and $j$, we deduce that
    \begin{equation}\label{below_aj}
        2^i\abs{A_i}\gtrsim \rho^\tau,\quad 2^j\abs{A_j}\gtrsim \rho^\tau
        \quad\mbox{and}\quad
        2^i\abs{A_i}^{\frac{1}{2}} \gtrsim \rho^\tau\norm{\mu_\rho}_2, \quad 2^j\abs{A_j}^{\frac{1}{2}} \gtrsim \rho^\tau\norm{\mu_\rho}_2.
    \end{equation}
    We record the following useful consequence of the estimates above:
    \begin{displaymath} \|\mu_{\rho}\|_{2} \stackrel{\eqref{below_aj}}{\lesssim} \rho^{-\tau}2^{i}|A_{i}|^{\frac{1}{2}} \stackrel{\eqref{above_ai}}{\lesssim} \rho^{-\tau} |A_{i}|^{-1}|A_{i}|^{\frac{1}{2}} = \rho^{-\tau} |A_{i}|^{-\frac{1}{2}}, \end{displaymath}
    so (also by the symmetry of $i$ and $j$)
    \begin{equation}\label{form26} \max\{|A_{i}|,|A_{j}|\} \lesssim \rho^{-2\tau}\|\mu_{\rho}\|_{2}^{-2} \stackrel{\eqref{eq:mularge}}{\leq} \rho^{1 - s - 3\tau}. \end{equation}

    \bigskip
    \noindent\underline{Step~2: Additive-multiplicative Balog-Szemerédi-Gowers}\\
    It follows from the definition \eqref{eq:def-Y0} that for every $y \in Y_0$,
    \begin{align*}
        \norm{\mathbf{1}_{x A_i} \mm \mathbf{1}_{yA_j}}^2_2
         & \geq \rho^{2\tau}\cdot 2^{-2i-2j}\cdot \norm{\mu_\rho}_2^2                                                 \\
         & \gtrsim \rho^{2\tau} \cdot 2^{-i}\cdot \abs{A_i}^{\frac{1}{2}} \cdot 2^{-j} \cdot \abs{A_j}^{\frac{1}{2}}.
    \end{align*}
    Since $2^i\abs{A_i}\lesssim 1$ by \eqref{above_ai}, and likewise for $A_j$, this yields
    \begin{equation}\label{energy}
        \norm{\mathbf{1}_{x A_i} \mm \mathbf{1}_{yA_j}}^2_2
        \gtrsim \rho^{2\tau} \cdot \abs{A_i}^{\frac{3}{2}} \cdot \abs{A_j}^{\frac{3}{2}}.
    \end{equation}
    The left-hand side is equal to the additive energy of $xA_i$ and $yA_j$ so, by the Balog-Szemerédi-Gowers lemma (see e.g.\cite[Theorem~5.2]{tao}), this implies that for every $y$ in $Y_0$, there exist sets $\bar{A}_y \subset xA_i$ and $\bar{A}'_y \subset A_j$ such that $\abs{\bar{A}_y}_\rho \gtrsim \rho^{2\tau}\abs{A_i}_\rho$ and $\abs{\bar{A}'_y}_\rho \gtrsim \rho^{2\tau}\abs{A_j}_\rho$ and
    \begin{equation}\label{17}
        \abs{\bar{A}_y - y\bar{A}_y'}_\rho \lesssim \rho^{-14\tau} \cdot \abs{A_i}_\delta^{\frac{1}{2}} \cdot \abs{A_j}_\delta^{\frac{1}{2}}
        \overset{\eqref{form26}}{\lesssim} \rho^{-17\tau-s}.
    \end{equation}
    Recall that $A_i$ is a union of $\rho$-intervals and $\mu_{\rho}$ is constant on $A_i$, up to a multiplicative factor. Using also that
    \begin{displaymath}
        \mu_\rho(A_i)\overset{\eqref{eq:mullAi}}{\gtrsim} 2^i\abs{A_i}\overset{\eqref{below_aj}}{\gtrsim}\rho^\tau,
    \end{displaymath}
    one has $\mu_\rho(\bar{A}_y)\gtrsim \rho^{3\tau}$. Since $\delta^{-\epsilon}\leq\rho^{-\tau}$ by assumption, we may use Lemma~\ref{lem:energy} to find a $(\rho,s,\rho^{-20\tau})$-set $A_y\subset \bar{A}_y$ with $\mu_\rho(A_y)\gtrsim\rho^{6\tau}$.
    Similarly, we may find a $(\rho,s,\rho^{-20\tau})$-set $A_y'\subset \bar{A}_y'$ with $\mu_\rho(A_y')\gtrsim\rho^{6\tau}$.
    One has in particular $\abs{A_y}_\rho\geq\rho^{-s+20\tau}$ and $\abs{A_y'}_\rho\geq\rho^{-s+20\tau}$, so \eqref{17} yields
    \[
        \abs{A_y - yA_y'}_\rho \lesssim \rho^{-40\tau}\cdot \abs{A_y}_{\rho}^{\frac{1}{2}}\cdot \abs{A_y'}_{\rho}^{\frac{1}{2}}, \qquad y \in Y_{0}.
    \]
    For subsets $A, A' \subset \R$ and $N\geq 1$, let us write $A \approx_{(N)} A'$ if
    \[
        \abs{A - A'}_\rho \lessapprox \rho^{-N\tau} \abs{A}_\rho^{1/2} \abs{A'}_\rho^{1/2}.
    \]
    For instance, the above bound on $\abs{A_y-yA_y'}_\delta$ can be rewritten
    \begin{equation}
        \label{eq:aAcApcb}
        A_y \approx_{(40)} yA'_y.
    \end{equation}
    Ruzsa's triangle inequality (see \cite[Proposition 3.5]{GKZ21} for the discretised version) can be summarised as: the relation $\approx$ is transitive
    i.e. $A \approx A'$ and $A' \approx A''$ implies $A' \approx A''$.
    More precisely,
    \begin{equation} \label{eq:ruzsa}
        ( A \approx_{(N)} A'\ \mbox{and}\ A'\approx_{(N')} A'' )
        \quad\implies\quad
        A \approx_{(N+N')} A''.
    \end{equation}
    Taking $\rho$-neighborhoods if necessary, we may assume that $\abs{A_y}_\rho\sim\rho^{-1}\abs{A_y}$ and similarly for $A'_y$.
    Write $X = xA_i \times A_j \subset [0,1]^{2}$ and $X_y = A_y \times A'_y \subset X$.
    Note that
    \[
        \abs{X_y}_\rho \gtrsim \rho^{20\tau} \abs{X}_\rho, \quad y\in Y_0.
    \]
    Let $\nu_0 := \nu(Y_{0})^{-1}\nu|_{Y_{0}}$. From Fubini and the Cauchy-Schwarz inequality applied to the function $u \mapsto \int_{Y} \mathbf{1}_{X_y}(u) \,d \nu_0(y)$, we infer that
    \begin{align*}
        \iint_{Y \times Y} \abs{X_y \cap X_z} \,d \nu_0(y)\, d\nu_0(z) & = \iint_{Y\times Y} \int\mathbf{1}_{X_y\cap X_z}(u)\,du \,d \nu_0(y) \,d \nu_0(z)               \\
                                                                       & \ge \left( \int_Y \int\mathbf{1}_{X_y}(u) \, du\, d\nu_0(y) \right)^2 \gtrsim \rho^{40\tau}|X|.
    \end{align*}
    The same inequality holds for the $\rho$-covering numbers. Hence, for some $y_\star \in Y_{0}$,
    \[
        \int_Y \abs{X_y \cap X_{y_\star}}_\rho\,d \nu_0(y) \gtrsim \rho^{40\tau} \abs{X}_\rho.
    \]
    Since also $\abs{X_y\cap X_{y_\star}}_\rho\leq\abs{X}_\rho$ for all $y\in Y_0$, we see that
    \[
        \nu_0(Y_1)\gtrsim \rho^{40\tau}\quad\text{where }Y_1=\left\{y :  \abs{X_y \cap X_{y_\star}}_\rho \gtrapprox \frac{1}{2}\rho^{40\tau}\abs{X_\rho}\right\}.
    \]
    Recalling \eqref{eq:nu-Y0}, we see that $\nu(Y_1)\gtrsim \rho^{41\tau}$. By Lemma~\ref{lem:energy} and the assumption that $\delta^{-\epsilon}\leq\rho^{-\tau}$, there exists a non-empty $(\rho,s,\rho^{-250\tau})$-set $Y\subset Y_1$.

    Abbreviate $A_{y_\star} =: A_\star$ and $A'_{y_\star} =: A'_\star$, thus $X_{y_{\star}} = A_{\star} \times A_{\star}'$.
    Then, for every $y \in Y$,
    \begin{equation}
        \label{eq:AstarCap}
        \abs{A_y \cap A_\star}_\rho \gtrapprox \rho^{40\tau}\abs{A_1}_\rho \,
        \text{ and }\,
        \abs{A'_y \cap A'_\star}_\rho \gtrapprox \rho^{40\tau}\abs{A_2}_\rho.
    \end{equation}
    For $y \in Y$, by \eqref{eq:aAcApcb},
    \[
        A_y \approx_{(40)} yA'_y \approx_{(40)} A_y
    \]
    so Ruzsa's triangle inequality in the form \eqref{eq:ruzsa} yields
    \[
        A_y \approx_{(80)} A_y.
    \]
    Using \eqref{eq:AstarCap} and the definition of the symbol $\approx$, we get $A_y \cap A_\star \approx_{(120)} A_y$, and for the same reason $A_y \cap A_\star \approx_{(120)} A_\star$.
    Hence, for $y\in Y$ we have
    \[
        A_\star \approx_{(120)} A_y\cap A_\star
        \approx_{(120)} A_y
        \approx_{(40)} yA'_y
        \approx_{(120)} y(A'_y \cap A'_\star)
        \approx_{(120)} yA'_\star,
    \]
    and therefore
    \begin{equation}
        \label{piysmall}
        \abs{\pi_y(A_\star\times A'_\star)}_\rho = \abs{A_\star-yA'_\star}_\rho
        \lessapprox \rho^{-520\tau } \cdot \abs{A_\star}_\rho^{\frac{1}{2}}\cdot\abs{A_\star'}_\rho^{\frac{1}{2}}.
    \end{equation}
    From $A_{\star} = A_{y_{\star}} \subset xA_{i}$ and \eqref{form26},
    \[
        \abs{A_\star}_\rho \lesssim \rho^{-1} \abs{A_i}
        \lesssim \rho^{-s - 3\tau},\]
    and similarly $\abs{A_\star'}_\rho\lesssim\rho^{-s-3\tau}$, so \eqref{piysmall} implies
    \begin{equation}\label{piysmall2}
        \abs{\pi_y(A_\star\times A'_\star)}_\rho \lesssim \rho^{-s-523\tau}.
    \end{equation}
    Let $c_0=\frac{1}{24}$ be the universal constant given by Theorem~\ref{thm:rescy}.
    Choosing $\tau = \frac{c_0}{524}t$, we find that $Y$, $A_\star$ and $A_\star'$ are $(\rho,s,\rho^{-c_0t})$-sets, so there exists $y \in Y$ such that
    \[
        \abs{\pi_y(A_\star\times A'_\star)}_\rho \geq \rho^{-s -c_0t}.
    \]
    This contradicts \eqref{piysmall2} and therefore \eqref{flatteningPi} must hold.
\end{proof}

\subsection{Proof of quantitative Fourier decay}

The conclusion of the proof of Theorem~\ref{second} uses the flattening lemma, the base case Proposition~\ref{basecase}, and the elementary Lemma~\ref{order} to reorganise additive and multiplicative convolutions.

\begin{proof}[Proof of Theorem~\ref{second}] Recall the hypotheses: $\mu_{1},\ldots,\mu_{n}$ are probability measures on $[1,2]$ satisfying $I_{\sigma}^{\delta}(\mu_{i}) \leq \delta^{-\epsilon}$ for all $1 \leq i \leq n$, where $\sigma > 0$, $n \geq C\sigma^{-1} \geq 2$ and $\epsilon \in (0,\epsilon(\sigma)]$. In fact, it turns out that the (absolute) constant $C \geq 1$ can be taken to be $C = 4C_{0}$, where $C_{0} \geq 1$ is the absolute constant provided by Lemma \ref{flattening}. With this notation, we will need
    \begin{equation}\label{form27} 0 < \epsilon < \epsilon(\sigma) := \frac{\sigma^{\ceil{C_{0}/\sigma}}}{43C^{\ceil{C_{0}/\sigma}}}. \end{equation}
    Under these hypotheses, we claim that
    \begin{displaymath} \abs{(\mu_1\ff\cdots\ff\mu_{n})^\wedge(\xi)} \leq \delta^{\tau}, \qquad |\xi| \sim \delta^{-1}, \end{displaymath}
    where $\tau \geq e^{-C\sigma^{-1}}$, and provided that $\delta > 0$ is sufficiently small.

    We assume that $\sigma < \tfrac{3}{4}$, since otherwise the conclusion follows from Proposition~\ref{basecase}. Write
    \begin{equation} \label{eq:def-ell}
        \ell := \ceil{C_{0}/\sigma} \leq n/2.
    \end{equation}
    Define $\Pi_1 = \mu_1$, and for $2 \leq k \leq \ell$,
    \[
        \Pi_k = (\Pi_{k-1}\ff\mu_k)\mm(\Pi_{k-1}\ff\mu_k).
    \]
    By a single application of Lemma \ref{flattening} to the measures $\mu_{1},\mu_{2}$, we find $I_{\sigma + \sigma/C_{0}}(\Pi_{2}) \leq \delta^{-C_{0}\epsilon/\sigma}$. Next, a second application Lemma \ref{flattening} to the measures $\Pi_{2}$ and $\mu_{3}$ shows that
    \begin{displaymath} I_{\frac{3\sigma}{C_{0}}}(\Pi_{3}) \lesssim I_{\sigma + 2\sigma/C_{0}}(\Pi_{3}) \leq \delta^{-C_{0}^{2}\epsilon/\sigma^{2}}, \end{displaymath}
    provided that $\delta > 0$ is small enough. Fix $\eta := \tfrac{1}{43} < \tfrac{1}{6} - \tfrac{1}{7}$. Continuing in this manner at most $\ell$ times and recalling the choice of $\epsilon$ in \eqref{form27}, we find
    \[
        I_{\frac{2}{3}}^{\delta}(\Pi_{\ell}) \leq I_{\frac{\ell\sigma}{C_{0}}}^\delta(\Pi_{\ell}) \leq \delta^{-C_{0}^{\ell}\epsilon/\sigma^{\ell}} \leq \delta^{-\eta},
    \]
    provided again that $\delta > 0$ is small enough.  Using the general inequality $\|\nu_{\delta}\|_{2}^{2} \lesssim \delta^{s - 1}I_{s}^{\delta}(\nu)$ for $0 < s < 1$, we infer that
    \begin{equation}\label{form28}
        I_{\frac{2}{3}}^\delta(\Pi_\ell) \leq \delta^{-\eta} \quad \Longrightarrow \quad \|(\Pi_{\ell})_{\delta}\|_{2}^{2} \lesssim \delta^{(2/3 - \eta) - 1}. \end{equation}
    For reasons to become apparently shortly, we then perform a similar construction for the measures $\{\mu_{\ell + 1},\ldots,\mu_{2\ell}\} \subset \{\mu_{1},\ldots,\mu_{n}\}$. Let $\Pi_1'=\mu_{\ell+1}$ and for $2 \leq k \leq \ell$,
    \[
        \Pi_k' = (\Pi_{k-1}'\ff\mu_{\ell + k})\mm(\Pi_{k-1}'\ff\mu_{\ell + k}).
    \]
    By exactly the same reasoning which led to \eqref{form28}, we now have
    \[
        \|(\Pi_{\ell}')_{\delta}\|_{2}^{2} \lesssim \delta^{(2/3 - \eta) - 1}.
    \]
    Proposition~\ref{basecase} now implies that for all $|\xi|\sim\delta^{-1}$,
    \[
        \abs{\widehat{\Pi_\ell \ff\Pi_\ell'}(\xi)} \lesssim \delta^{\frac{(4/3 - 2\eta) - 1}{2}} = \delta^{\tfrac{1}{6} - \eta} < \delta^{\frac{1}{7}},
    \]
    since $\eta < \tfrac{1}{6} - \tfrac{1}{7}$. Applying Lemma~\ref{order} to $\mu = \Pi_{\ell-1}\ff\mu_\ell$ and $\nu=\Pi_\ell'$, we get
    \[
        \abs{((\Pi_{\ell-1}\ff\mu_{\ell}) \ff\Pi_\ell')^\wedge(\xi)}^2 \leq \abs{\widehat{\Pi_\ell\ff\Pi_\ell'}(\xi)} \leq \delta^{\frac{1}{7}}
    \]
    and repeating this argument $\ell-1$ times,
    \[
        \abs{(\mu_1\ff\cdots\ff\mu_{\ell}\ff\Pi_\ell')^\wedge(\xi)}^{2^{\ell-1}} \leq \delta^{\frac{1}{7}}.
    \]
    Lemma~\ref{order} applied again $\ell-1$ times for $\Pi_\ell'$ yields
    \[
        \abs{(\mu_1\ff\cdots\ff\mu_{\ell}\ff\mu_{\ell+1}\ff\cdots\ff\mu_{2\ell})^\wedge(\xi)}^{4^{\ell-1}} \leq \delta^{\frac{1}{7}}
    \]
    and thus
    \begin{equation}\label{form29}
        \abs{(\mu_1\ff\cdots\ff\mu_{2\ell})^\wedge(\xi)} \leq \delta^{2^{-(2\ell+1)}}.
    \end{equation}
    Writing $(\mu_1\ff\cdots\ff\mu_n)^\wedge(\xi)$ as an average of $(\mu_1\ff\cdots\ff\mu_{2\ell})^\wedge(a\xi)$, for $a \in \spt(\mu_{2\ell+1}\ff\cdots\ff\mu_n) \subset [1,\infty)$, the same bound holds for $\abs{(\mu_1\ff\cdots\ff\mu_n)^\wedge(\xi)}$. Recalling \eqref{eq:def-ell} and letting $\tau = 2^{-(2\ell+1)} \geq 2^{-3\ell} \geq 2^{-4C_0\sigma^{-1}}$, we conclude that
    \begin{equation*}
        \abs{(\mu_1\ff\dots\ff\mu_{n})^\wedge(\xi)} \leq \delta^{\tau}, \qquad |\xi| \sim \delta^{-1}.
    \end{equation*}

\end{proof}


\end{document}